\documentclass[a4paper,12pt]{article}
\usepackage{amssymb,amsmath,MnSymbol,epsfig,amscd,eucal,psfrag,amsthm,enumerate,tikz-cd}
\usepackage[T1]{fontenc}
\usepackage[latin9]{inputenc}
\usepackage{graphicx,color}

\newtheorem{theorem}{Theorem}[section]

\newtheorem{lemma}[theorem]{Lemma}
\newtheorem{proposition}[theorem]{Proposition}

\newtheorem{letterthm}{Theorem}
\newtheorem{lettercor}[letterthm]{Corollary}

\makeatletter
\newtheorem*{rep@theorem}{\rep@title}
\newcommand{\newreptheorem}[2]{%
\newenvironment{rep#1}[1]{%
 \def\rep@title{#2 \ref{##1}}%
 \begin{rep@theorem}}%
 {\end{rep@theorem}}}
\makeatother

\newreptheorem{theorem}{Theorem}
\newreptheorem{cor}{Corollary}

\theoremstyle{definition}
\newtheorem{definition}[theorem]{Definition}

\theoremstyle{remark}
\newtheorem{remark}[theorem]{Remark}
\newtheorem{example}[theorem]{Example}

\newcommand{\Q}{\mathbb{Q}}
\newcommand{\R}{\mathbb{R}}

\newcommand{\Z}{\mathbb{Z}}

\newcommand{\Lk}{\mathrm{Lk}}

\newcommand{\Area}{\mathrm{Area}}
\newcommand{\Ess}{\mathrm{Ess}}
\newcommand{\Irred}{\mathrm{Irred}}
\newcommand{\Orig}{\mathrm{Orig}}

\newcommand{\Surf}{\mathrm{Surf}}

\newcommand{\eblocks}{\curlyE}

\newcommand{\vblocks}{\curlyV}

\newcommand{\curlyE}{\mathcal{E}}

\newcommand{\curlyV}{\mathcal{V}}

\newcommand{\into}{\hookrightarrow}

\newcommand{\opp}[1]{{#1}^*}

\usepackage{hyperref}

\input xy
\xyoption{all}

\title{Rationality theorems for curvature invariants of 2-complexes}
\author{Henry Wilton}

\newcommand{\Addresses}{{% additional braces for segregating \footnotesize
  \bigskip
  \footnotesize

  \textsc{DPMMS, Centre for Mathematical Sciences, Wilberforce Road, Cambridge, CB3 0WB, UK}\par\nopagebreak
  \textit{E-mail address:} \texttt{h.wilton@maths.cam.ac.uk}

}}

\begin{document}

\maketitle

\begin{abstract}
Let $X$ be a finite, 2-dimensional cell complex. The curvature invariants $\rho_\pm(X)$ and $\sigma_\pm(X)$ were defined in \cite{wilton_rational_2024}, and a programme of conjectures was outlined. Here, we prove the foundational result that the quantities $\rho_\pm(X)$ and $\sigma_\pm(X)$ are the extrema of explicit rational linear-programming problems. As a result they are rational, realised, and can be computed algorithmically.
\end{abstract}

Throughout this paper, we will be interested in a finite, 2-dimensional cell complex $X$. Perhaps $X$ can be simplified in a trivial way, either because some face of $X$ has a free edge, or because some vertex of $X$ is locally separating. More generally, even if $X$ does not have these features, either feature may appear after modifying the 1-skeleton of $X$ by a homotopy equivalence. In this case $X$ is called \emph{reducible}, and so we will always assume that $X$ is irreducible\footnote{See Definition \ref{def: (Visibly) irreducible 2-complex} for the precise definition of an irreducible 2-complex.}.

The \emph{average curvature} of $X$, defined to be
\[
\kappa(X):=\frac{\chi(X)}{\Area(X)}\,,
\]
where $\chi(X)$ is the Euler characteristic and $\Area(X)$ is the number of 2-cells, provides a crude measure of the curvature of $X$. Four more refined curvature invariants of $X$ are proposed in \cite{wilton_rational_2024}. The idea is to probe $X$ by measuring the average curvatures of 2-complexes $Y$ that map combinatorially to $X$. In fact, it is fruitful to endow $Y$ with some extra structure -- a notion of area on the 2-cells -- and to allow the map $Y\to X$ to branch over the centres of the 2-cells. This leads to the notion of a \emph{branched} 2-complex, and the average curvature $\kappa(Y)$ extends naturally to this setting.\footnote{See Definitions \ref{def: Branched 2-complex} and \ref{def: Total and average curvatures} for the definitions of a branched 2-complex and its average curvature.}

The \emph{maximal irreducible curvature} of $X$ is then defined to be
\[
\rho_+(X):=\sup_{Y\in\Irred(X)}\kappa(Y)\,,
\]
where $\Irred(X)$ consists of all finite, irreducible, branched 2-complexes $Y$ equipped with an essential map to $X$.\footnote{See Definition \ref{def: Essential map} for the definition of an essential map.} The corresponding infimum, $\rho_-(X)$, is called the \emph{minimal irreducible curvature}.   Our first main theorem asserts that the extrema in the definitions of $\rho_+$ and $\rho_-$ can be computed using an explicit rational linear-programming problem.

\begin{letterthm}[Rationality theorem for irreducible curvatures]\label{introthm: Rationality for irreducible curvature}
If $X$ is a finite, 2-dimensional cell complex and $\Irred(X)$ is non-empty, then:
\begin{enumerate}[(i)]
\item the curvature invariants $\rho_+(X)$ and $\rho_-(X)$ are the maximum and minimum, respectively, of an explicit rational linear-programming problem; furthermore,
\item $\rho_+(X)$ is attained by some $Y_{\max}\in\Irred(X)$ and $\rho_-(X)$ is attained by some $Y_{\min}\in\Irred(X)$.
\end{enumerate}
In particular, $\rho_\pm(X)\in\Q$ and both quantities can be computed algorithmically from $X$.
\end{letterthm}

As long as $X$ itself is irreducible, $\Irred(X)$ is non-empty because $X\in\Irred(X)$. See also \cite[Lemma 3.10]{louder_uniform_2024} or \cite[Lemma 4.5]{wilton_rational_2024} for weaker conditions under which $\Irred(X)$ is non-empty. By convention,  $\rho_+(X)=-\infty$ and $\rho_-(X)=\infty$ if $\Irred(X)$ is empty. 

Theorem \ref{introthm: Rationality for irreducible curvature} plays a foundational role in a programme to study 2-complexes outlined in \cite{wilton_rational_2024}. The reader is referred to that paper for motivation and applications of this theorem, as well as its companion Theorem \ref{introthm: Rationality for surface curvature} below. We describe a few applications next, although many further applications are conjectured in \cite{wilton_rational_2024}.

If $\rho_+(X)\leq 0$ then $X$ is said to have \emph{non-positive irreducible curvature}, and likewise if $\rho_+(X)< 0$ then $X$ is said to have \emph{negative irreducible curvature}. Both of these curvature bounds have consequences for the topology of $X$ and for the structure of its fundamental group: non-positive irreducible curvature implies that $X$ is aspherical,  $\pi_1(X)$ is locally indicable, and every finitely generated subgroup of $\pi_1(X)$ has finite second Betti number.  Negative irreducible curvature implies that $\pi_1(X)$ is 2-free, coherent, and every one-ended finitely generated subgroup is co-Hopf and large. Thus, Theorem \ref{introthm: Rationality for irreducible curvature} demonstrates that the invariant $\rho_+(X)$ provides an effective sufficient condition for these properties. 

Upper bounds on irreducible curvature are closely related to the curvature properties for 2-complexes used by Wise \cite{helfer_counting_2016,wise_coherence_2022} and Louder--Wilton \cite{louder_stackings_2017,louder_one-relator_2020,louder_negative_2022,louder_uniform_2024} to study one-relator groups. Non-positive irreducible curvature implies Wise's \emph{non-positive immersions} property, while negative irreducible curvature implies the \emph{(uniform) negative immersions} property studied by Louder--Wilton. In \cite[Problem 1.8]{wise_coherence_2022}, Wise asked whether non-positive immersions can be recognised algorithmically. Theorem \ref{introthm: Rationality for irreducible curvature} does not answer Wise's question, since there are 2-complexes that have non-positive immersions but do not have non-positive irreducible curvature \cite{blufstein_2-complex_2023}. Nevertheless, Theorem \ref{introthm: Rationality for irreducible curvature} can be thought of as providing a positive answer to Wise's question for close cousins of the non-positive and negative immersions conditions. 

As well as implying that the curvature bounds $\rho_\pm(X)$ are computable, Theorem \ref{introthm: Rationality for irreducible curvature} also asserts that they are \emph{realised}, and this fact is expected to play a role in the deepest applications of this work. A similar realisation result is at the heart of the main theorem of \cite{louder_uniform_2024}, that one-relator groups with negative immersions are coherent. The realisation of $\rho_-(X)$ is also used in the remarkable fact that irreducible curvature nearly characterises surfaces among branched 2-complexes.  An easy argument (essentially the Riemann--Hurwitz theorem) implies that, if the realisation of $X$ is homeomorphic to a surface, then $X$ has \emph{constant} irreducible curvature, meaning that $\rho_+(X)=\rho_-(X)$. The following corollary of Theorem \ref{introthm: Rationality for irreducible curvature} provides a partial converse.

\begin{lettercor}\label{introcor: Constant irreducible curvature}
If $X$ is an irreducible branched 2-complex and $\rho_+(X)=\rho_-(X)$ then $X$ is irrigid. In particular, there is an essential map from a surface $S\to X$ such that $\kappa(S)=\kappa(X)$.
\end{lettercor}

The reader is referred to \cite[Theorem D]{wilton_rational_2024} for the proof of Corollary \ref{introcor: Constant irreducible curvature}, as well as the definition of an irrigid 2-complex. Besides Theorem \ref{introthm: Rationality for irreducible curvature}, the techniques of \cite{wilton_one-ended_2011,wilton_essential_2018} provide the main ingredients of the proof.

Since any branched 2-complex $Y$ can be simplified to make it either empty or irreducible, $\Irred(X)$ is the largest class of complexes mapping to $X$ that we can usefully consider. At the other extreme, the smallest class of (branched) 2-complexes that we can usefully consider consists of closed surfaces.  This motivates the definition of the \emph{maximal surface curvature}
\[
\sigma_+(X):=\sup_{Y\in\Surf(X)}\kappa(Y)\,,
\]
where  $\Surf(X)$ consists of all branched 2-complexes $Y$ with realisations homeomorphic to closed surfaces, equipped with an essential map to $X$. Again, the corresponding infimum, $\sigma_-(X)$, is called the \emph{minimal surface curvature}. Our second main theorem is the analogue of Theorem \ref{introthm: Rationality for irreducible curvature} for $\sigma_+(X)$ and $\sigma_-(X)$.

\begin{letterthm}[Rationality theorem for surface curvatures]\label{introthm: Rationality for surface curvature}
If $X$ is a finite, 2-dimensional cell complex and $\Surf(X)$ is non-empty, then:
\begin{enumerate}[(i)]
\item the curvature invariants $\sigma_+(X)$ and $\sigma_-(X)$ are the maximum and minimum, respectively, of an explicit rational linear-programming problem; furthermore,
\item $\sigma_+(X)$ is attained by some $Y_{\max}\in\Surf(X)$ and $\sigma_-(X)$ is attained by some $Y_{\min}\in\Surf(X)$.
\end{enumerate}
In particular, $\sigma_\pm(X)\in\Q$ and can be computed algorithmically from $X$.
\end{letterthm}

The question of when $\Surf(X)$ is empty is more subtle than for $\Irred(X)$ since, on the face of it, there could be irreducible 2-complexes with no essential maps from surfaces. However, the main theorem of \cite{wilton_essential_2018} implies that $\Surf(X)$ is non-empty whenever $\Irred(X)$ is, in particular, whenever $X$ is itself irreducible.  (See also Remark \ref{rem: Relation to essential surfaces for graph pairs} below.) Again, by convention, $\sigma_+(X)=-\infty$ and $\sigma_-(X)=\infty$ if $\Surf(X)$ is empty.

Theorems \ref{introthm: Rationality for irreducible curvature} and \ref{introthm: Rationality for surface curvature} are in fact both special cases of the more general Theorem \ref{thm: General rationality theorem}, which proves rationality for any invariants defined by essential maps from branched 2-complexes defined by certain conditions on the links of vertices.

Alongside $\rho_\pm(X)$, the invariants $\sigma_\pm(X)$ also play central roles in the programme described in \cite{wilton_rational_2024}, and the reader is again referred to that paper for further details and motivation.  We record one especially important consequence here: contrary to initial appearances, the above definitions only give three invariants rather than four.

\begin{lettercor}\label{introcor: Minimal curvatures are equal}
For any finite 2-complex $X$, $\rho_-(X)=\sigma_-(X)$.
\end{lettercor}

See \cite[Theorem A]{wilton_rational_2024} for the proof, which is similar to Corollary \ref{introcor: Constant irreducible curvature}.

\begin{remark}\label{rem: Relation to essential surfaces for graph pairs}
As a special case, Corollary \ref{introcor: Minimal curvatures are equal} implies the fact mentioned above that $\Surf(X)$ is empty if and only if $\Irred(X)$ is empty: $\Irred(X)$ is empty if and only if $\rho_-(X)=\infty$ while $\Surf(X)$ is empty if and only if $\sigma_-(X)=\infty$.
\end{remark}

Rationality theorems have played important roles in geometric group theory over the last two decades, starting with Calegari's proof that stable commutator length is rational in free groups \cite{calegari_stable_2009}\footnote{See \cite{wilton_rational_2024} for a discussion of the relationship between stable commutator length in free groups and the surface curvatures of 2-complexes}, and continuing with \cite{wilton_essential_2018} and \cite{louder_uniform_2024}, in which they were applied to analyse the subgroup structures of various classes of finitely presented groups.  All these rationality theorems deal with families of maps $Y\to X$, and one major difficulty is to ensure that the maps being considered are homotopically non-trivial.

Calegari's maps in \cite{calegari_stable_2009} are automatically homotopically non-trivial because they represent a non-trivial second relative homology class.  The rationality theorems of \cite{wilton_essential_2018} and \cite{louder_uniform_2024} apply to \emph{face-essential maps} (in the terminology of \cite{louder_uniform_2024}). 

Here, we are concerned with the more restrictive class of \emph{essential} maps\footnote{Again, see Definition \ref{def: Essential map} for a precise definition.}, which are required to be injective on fundamental groups of 1-skeleta. We therefore need a method of locally recognising $\pi_1$-injective maps of graphs.

Such a method is provided by Stallings' famous observation that immersions of graphs are $\pi_1$-injective \cite{stallings_topology_1983}. Furthermore, Stallings' folding lemma (Lemma \ref{lem: Stallings folding lemma} below) asserts that any map of graphs can be folded to an immersion. However, folding the 1-skeleton usually obscures the structure of a 2-complex. For an example, consider a 2-complex $Y'$ which is a fine triangulation of a surface, and then let $Y$ be the result of folding a pair of edges in $Y'$ with one common endpoint. To see that $Y$ is homotopy equivalent to a surface, we need to  \emph{unfold} to $Y'$. Because of these kinds of considerations, we need a criterion to ensure that a map of graphs $\Delta\to\Gamma$ is $\pi_1$-injective, which does not require the map to be an immersion.

The most important technical innovation of this paper is the notion of an \emph{origami}, which provides the needed criterion. An origami $\Omega$ on a graph $\Delta$ can be thought of as a certain kind of topological graph of graphs, with (some of) the vertices of the vertex-spaces equal to the vertex-set of $\Delta$ -- see Definition \ref{def: Origami} for a precise definition. There is a natural combinatorial $\pi_1$-surjection from $\Delta$ to the underlying graph $\Delta/\Omega$ of this graph of graphs.  The idea behind the definition of an origami is that it keeps track of a sequence of folds that compose to give the map $\Delta\to\Delta/\Omega$.

More importantly, an origami can also detect whether the associated map is $\pi_1$-injective. An origami $\Omega$ is called \emph{essential} if every vertex-space of $\Omega$ is a tree, and is said to be \emph{compatible} with a map of graphs $f:\Delta\to\Gamma$ if $f$ factors through $\Delta/\Omega$, and the induced map $\Delta/\Omega\to\Gamma$ is an immersion. The next theorem, which is the most important theorem about origamis, encapsulates the fact that they can be used to certify that maps of graphs are $\pi_1$-injective. I believe it to be of independent interest.

\begin{letterthm}\label{thm: Essential origamis}
A map of finite, connected, core graphs $f:\Gamma\to\Delta$ is injective on fundamental groups if and only if there is an essential origami $\Omega$ on $\Delta$ that is compatible with $f$.
\end{letterthm}

The first two sections of the paper are concerned entirely with graphs. In \S\ref{sec: Graphs} we recall the basic formalism of graphs, especially Serre graphs, which will be our primary concern. Origamis are introduced and studied in \S\ref{sec: Origamis}, leading up to the proof of Theorem \ref{thm: Essential origamis}. The basics of 2-dimensional cell complexes, including branched 2-complexes and essential maps, are introduced in \S\ref{sec: Essential maps of 2-complexes}. The rational cone $C(\R)$ that underlies the proofs of Theorems \ref{introthm: Rationality for irreducible curvature} and \ref{introthm: Rationality for surface curvature} is constructed in \S\ref{sec: A linear system}. Finally, in \S\ref{sec: Curvature invariants}, the invariants $\rho_\pm(X)$ and $\sigma_\pm(X)$ are defined, and it is shown that they can be seen as extremal values of a projective function $\kappa$ on $C(\R)$, which completes the proof.

\subsection*{Acknowledgements}

Thanks are due, as ever, to Lars Louder for useful conversations.  I am also very grateful to Doron Puder and Niv Levhari for pointing out an error in an earlier version. Finally, thanks to the anonymous referee, whose comments have greatly improved the exposition.

\section{Graphs}\label{sec: Graphs}

This section introduces the graphs that will play a role in the argument. The material is all standard, but the details of the definitions will be important in the new material on origamis in the next section.

We start with the simplest kind of graph. A \emph{directed graph} $\Gamma$ consists of a vertex set $V_\Gamma$, an edge set $E_\Gamma$, an \emph{initial map} $\iota=\iota_\Gamma: E_\Gamma\to V_\Gamma$, and a \emph{terminus map} $\tau=\tau_\Gamma:E_\Gamma\to V_\Gamma$.

Most of our graphs will be Serre graphs. Following Serre \cite{serre_arbres_1977} and Stallings \cite{stallings_topology_1983}, a Serre graph, or usually just a \emph{graph}, is a directed graph $\Gamma$ together with a fixed-point-free involution $e\mapsto\opp{e}$ on $E_\Gamma$, such that
\[
\tau(e)=\iota(\opp{e})
\]
for all $e\in E_\Gamma$. The elements $e$ of $E_\Gamma$ should be regarded as \emph{oriented} edges of $\Gamma$, while an \emph{unoriented} or \emph{geometric} edge is a pair $\{e,\opp{e}\}$. 

The reader is referred to the paper of Stallings \cite{stallings_topology_1983} for further facts, constructions and terminology, but we adapt and synthesise a few important notions here.

The \emph{link} (which Stallings calls the \emph{star}) of a vertex $v$ of $\Gamma$ is the subset
\[
\Lk_\Gamma(v)=\iota^{-1}(v)\,.
\]
A morphism of graphs $f:\Delta\to\Gamma$ naturally induces a map of links $df_v:\Lk_\Delta(v)\to\Lk_\Gamma(f(v))$. The morphism $f$ is an \emph{immersion} if $df_v$ is injective for all $v$, and in this case the induced map on fundamental groups is injective \cite[Proposition 5.3]{stallings_topology_1983}. A famous lemma of Stallings provides the basis for analysing morphisms of graphs in terms of their effects on fundamental groups \cite[\S5.4]{stallings_topology_1983}.

\begin{lemma}[Stallings' folding lemma]\label{lem: Stallings folding lemma}
Any morphism of finite graphs $f:\Delta\to \Gamma$ factors as  
\[
\Delta\stackrel{f_0}{\to}\overline{\Delta}\stackrel{\bar{f}}{\to}\Gamma
\]
where $\bar{f}$ is an immersion. Furthermore, the map $f_0$ factors as a finite sequence of folds.
\end{lemma}

We recall the definition of a fold, and introduce some notation and terminology.

\begin{definition}[Fold]\label{def: Fold}
Let $\Delta$ be a graph and $a_1,a_2$ a pair of edges with $u=\iota(a_1)=\iota(a_2)$. The corresponding \emph{fold} is the quotient morphism 
\[
f_a:\Delta\to\Delta'
\]
where $\Delta'$ is obtained by identifying the edges $a_1$ and $a_2$ to a single edge $a$. This forces $\opp{a}_1$ and $\opp{a}_2$ to both be identified with $\opp{a}$, but the remaining edges of $\Delta$ correspond bijectively with the remaining edge of $\Delta'$. See Figure \ref{fig: Fold}. The fold $f_a$ is called \emph{essential} if $\tau(a_1)\neq\tau(a_2)$ or, equivalently, if $f_a$ is a homotopy equivalence.
\end{definition}

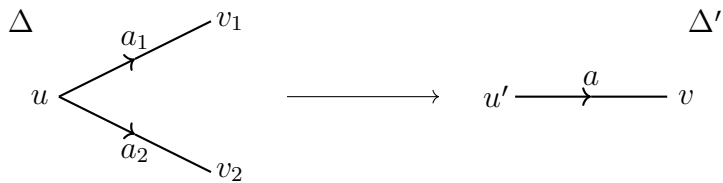
\begin{figure}[htp]
\begin{center}
\begin{tikzpicture}
\coordinate (u) at (-4,0);
\coordinate (v1) at (-2,1);
\coordinate (v2) at (-2,-1);
\coordinate (a1) at (-3,0.5);
\coordinate (a2) at (-3,-0.5);

\coordinate (u') at (2,0);
\coordinate (v) at (4,0);
\coordinate (a) at (3,0);

\coordinate (ulabel) at (-4.25,0);
\coordinate (v1label) at (-1.75,1);
\coordinate (v2label) at (-1.75,-1);
\coordinate (a1label) at (-3,0.75);
\coordinate (a2label) at (-3,-0.75);

\coordinate (u'label) at (1.75,0);
\coordinate (vlabel) at (4.25,0);
\coordinate (alabel) at (3,0.25);

\coordinate (Deltalabel) at (-4.5,1);
\coordinate (Delta'label) at (4.5,1);

\draw[thick,->] (u)-- (a1);\draw[thick,-] (a1) -- (v1);
\draw[thick,->] (u) -- (a2);\draw[thick,-] (a2) -- (v2);

\draw[->] (-1,0) -- (1,0);

\draw[thick,->] (u') -- (a);\draw[thick,-] (a) -- (v);

\node at (ulabel) {$u$};
\node at (v1label) {$v_1$};
\node at (v2label) {$v_2$};
\node at (a1label) {$a_1$};
\node at (a2label) {$a_2$};
\node at (u'label) {$u'$};
\node at (vlabel) {$v$};
\node at (alabel) {$a$};
\node at (Deltalabel) {$\Delta$};
\node at (Delta'label) {$\Delta'$};
\end{tikzpicture}
\end{center}
\caption{A Stallings fold $\Delta\to\Delta'$. Note that we do not in general insist that $v_1\neq v_2$.}
\label{fig: Fold}
\end{figure}

We say that $\Delta'$ is constructed from $\Delta$ by \emph{folding}, and likewise $\Delta'$ is constructed from $\Delta$ by \emph{unfolding}. It is useful to notice that unfolding is characterised by a certain partition of a link of a vertex.

\begin{remark}\label{rem: Unfolding and partitions}
Consider an essential fold $f_a:\Delta\to \Delta'$ as above, and write $v'=\tau(a)$, $v_1=\tau(a_1)$ and $v_2=\tau(a_2)$. The map $f_a$ induces a bijection between $(\Lk_{\Delta}(v_1)\smallsetminus \{\opp{a}_1\}) \sqcup (\Lk_{\Delta}(v_2))\smallsetminus \{\opp{a}_2\})$ and $\Lk_{\Delta'}(v')\smallsetminus\{\opp{a}\}$, and hence induces a partition of $\Lk_{\Delta'}(v')\smallsetminus\{\opp{a}\}$ into two. Conversely, a partition of $\Lk_{\Delta'}(v')\smallsetminus\{\opp{a}\}$ into two provides exactly the data needed to construct the graph $\Delta$ and the fold map $f_a$. Thus, the data of an essential unfolding of the edge $a$ in $\Delta'$ is equivalent to a partition of $\Lk_{\Delta'}(v')\smallsetminus\{\opp{a}\}$ into two sets. 
\end{remark}

One of the uses of Stallings' folding lemma is to construct cores for covering spaces of graphs. We next recall the relevant definitions.

\begin{definition}[Cores and core graphs]\label{def: Cores}
Let $\Gamma$ be a (not necessarily finite) graph. A \emph{core} for $\Gamma$ is a finite subgraph $\Gamma_0$ such that the inclusion map $\Gamma_0\into\Gamma$ is a homotopy equivalence. The graph $\Gamma$ itself is said to be a \emph{core graph} if no proper subgraph is a core for $\Gamma$. It is easy to see that a finite graph is core if and only if it has no vertices of valence 1.
\end{definition}

The power of Stallings' folding lemma lies in the fact that it provides local certificates -- the induced maps on links -- to certify that a morphism of graphs is injective on fundamental groups. However,  we will need to locally certify that a graph morphism $\Delta\to\Gamma$ is injective on fundamental groups while simultaneously remembering the structure of $\Delta$, which presents a problem since the $\pi_1$-injectivity of the folding map $\Delta\to\overline{\Delta}$ cannot be certified on any finite ball in $\Delta$. (Consider, for instance, the map of graphs that folds a circle $\Delta$ with $2n$ edges to a path with $n$ edges.) This problem is solved by origamis, which are the subject of the next section.

Although origamis will be defined by certain equivalence relations, it will be important to think of them as graphs of spaces, in the spirit of Scott and Wall \cite{scott_topological_1979}.  A \emph{graph of spaces} $X$ consists of a (possibly disonnected) vertex space $V_X$, a (necessarily disconnected) edge space $E_X$ equipped with an involution $\cdot\mapsto\opp{\cdot}$ that induces a fixed-point free involution of $\pi_0(E_X)$, and a continuous map $\iota:E_X\to V_X$. The \emph{underlying graph} $\Gamma_X$ of $X$ is obtained by applying the $\pi_0$ functor that takes a space to its set of path components. Adjointly, every graph is a graph of spaces with discrete vertex and edge sets. The \emph{geometric realisation} of $X$ is defined to be
\[
\left(V_X\sqcup (E_X\times [-1,1])\right)/\sim
\]
where $(y,-1)\sim \iota(y)$ and $(y,t)\sim (\opp{y},1-t)$ for all $y\in E_X$.  We will often abuse notation and denote the geometric realisation of $X$ by $X$.

\section{Origamis}\label{sec: Origamis}

In this section we introduce \emph{origamis}, which can provide a local certificate that a morphism of graphs $\Delta\to\Gamma$, which may not be an immersion, is $\pi_1$-injective. They are called `origamis' because they tell you how to fold.

\begin{definition}[Origami]\label{def: Origami}
Let $\Delta$ be a (finite, Serre) graph.  An \emph{origami} $\Omega$ on $\Delta$ is defined by an equivalence relation $\sim_O$ on the edge set $E_\Delta$. We shall call this relation the \emph{open} relation.  The open relation in turn defines the \emph{closed} relation $\sim_C$ on $E_\Delta$ by
\[
e_1\sim_C e_2 \Leftrightarrow \opp{e}_1\sim_O \opp{e}_2\,,
\]
and these two relations in turn allow us to define two auxiliary (directed) graphs. The \emph{edge graph} $E_\Omega$ is a bipartite directed graph, with the following two kinds of vertices:
\begin{enumerate}[(i)]
\item the $\sim_O$-equivalence classes of $E_\Delta$, denoted by $[e]_O$ (called the \emph{open vertices}); and
\item the $\sim_C$-equivalence classes of $E_\Delta$, denoted by $[e]_C$ (called the \emph{closed vertices}).
\end{enumerate}
The edge set of $E_\Omega$ is just $E_\Delta$, the edge set of $\Delta$ .  Each edge $e$ has initial vertex $[e]_O$ and terminal vertex $[e]_C$.

The \emph{vertex graph} $V_\Omega$ is also a bipartite directed graph, with the following two kinds of vertices:
\begin{enumerate}[(i)]
\item the vertex set $V_\Delta$ (called the $\Delta$-vertices);
\item the $\sim_C$-equivalence classes of $E_\Delta$, denoted by $[e]_C$ (again called the \emph{closed} vertices).
\end{enumerate}
The edge set of $V_\Omega$ is $E_\Delta$; the initial vertex of $e\in E_\Delta$ is the $\Delta$-vertex $\iota_\Delta(e)\in V_\Delta$, and its terminal vertex is $[e]_C$.

So far, these are merely notations defined by the equivalence relation $\sim_O$.   To define an origami, three conditions are required on the equivalence relation.
\begin{enumerate}[(a)]
\item\label{item: Non-singularity} \emph{Non-singularity:} For each $e\in E_\Delta$, the edges $e$ and $\opp{e}$ are in different connected components of $E_\Omega$.
\item\label{item: Consistency} \emph{Global consistency:} If $e_1\sim_O e_2$ then $\iota_\Delta(e_1)$ and $\iota_\Delta(e_2)$ are in the same component of $V_\Omega$.
\item\label{item: Edge forest} \emph{Local consistency:} If $e_1,e_2\in [e]_O$ then $e$ does not separate $\iota_\Delta(e_1)$ from $\iota_\Delta(e_2)$ in $V_\Omega$. That is, $\iota_\Delta(e_1)$ and $\iota_\Delta(e_2)$ are in the same component of $V_\Omega\smallsetminus e$.
\end{enumerate} 
If all three of these conditions are satisfied, then $\sim_O$ is said to define an \emph{origami} $\Omega$.

We will be primarily interested in origamis that also satisfy an additional property.

\begin{enumerate}[(a)]
\setcounter{enumi}{3}
\item\label{item: Essential} \emph{Essentiality:} If  the vertex graph $V_\Omega$ and edge graph $E_\Omega$ are forests then $\Omega$ is called \emph{essential}.
\end{enumerate}
\end{definition}

In fact, an origami $\Omega$ naturally defines a graph of spaces $X_\Omega$, with vertex space the realisation of $V_\Omega$ and edge space the realisation of $E_\Omega$, given by continuous maps $\iota_\Omega:E_\Omega\to V_\Omega$ and $\opp{\cdot}:E_\Omega\to E_\Omega$.

\begin{remark}\label{rem: Origami graph of spaces}
The initial map $\iota_\Omega:E_\Omega\to V_\Omega$ is defined non-canonically as follows.

The closed vertices of $E_\Omega$ are also vertices of $V_\Omega$, so we can canonically define $\iota_\Omega$ to be the identity on these; that is, $\iota_\Omega([e]_C)=[e]_C$. There is no canonical choice for the image of an open vertex $[e]_O$, but we may set $\iota_\Omega([e]_O)$ to be the $\Delta$-vertex $\iota_\Delta(e_0)$ of $V_\Omega$ for any fixed choice of representative $e_0\in[e]_O$. Evidently this depends on the choice of representative, but note that the global consistency condition implies that all such choices are in the same path component of $V_\Omega$.

It remains to define $\iota_\Omega$ on an edge $e$ of $E_\Omega$.  Since $[e]_C$ is joined to $\iota_\Delta(e)$ by the edge $e$, and consistency tells us that $\iota_\Delta(e)$ is in the same path component as $\iota_\Delta(e_0)$ for the chosen representative $e_0\in [e]_O$, it follows that there is a choice of path $\gamma_e$ in $V_\Omega$ from $\iota_\Omega([e]_C)=[e]_C$ to $\iota_\Omega([e]_O)=\iota_\Delta(e_0)$. Therefore, we may extend $\iota_\Omega$ continuously across $E_\Omega$ by  parametrising it as $\gamma_e$ on $e$.

To complete the definition of the graph of spaces structure $X_\Omega$, we need to define the involution $\opp{\cdot}$ on $E_\Omega$. This naturally extends the orientation-reversing involution on $E_\Delta$: thought of as an edge of $E_\Omega$, $e$ is sent to $\opp{e}$, while $[e]_O$ is sent to $[\opp{e}]_C$ and $[e]_C$ is sent to $[\opp{e}]_O$. (Note that the latter is well defined by the definition of the closed equivalence relation.) Finally, note that non-singularity of $\Omega$ implies that $\opp{\cdot}$ fixes no elements of $\pi_0(E_\Omega)$, so $X_\Omega$ is indeed  a graph of spaces. 
\end{remark}

\begin{remark}\label{rem: Initial vertex well defined}
Let $e$ be an edge of $\Delta$. Since $\iota_\Omega(e)=\iota_{V_\Omega}(e)=\iota_\Delta(e)$, we will use the notation $\iota(e)$ to denote all of these, without fear of confusion.
\end{remark}

For our purposes, the most important feature of this graph of spaces is the underlying graph.

\begin{definition}[Quotient graph and equivalence]\label{def: Quotient graph}
The underlying graph of $X_\Omega$ is called the \emph{quotient graph} of $\Omega$, and denoted by $\Delta/\Omega$. The quotient notation is justified by the existence of a quotient map $q:\Delta\to\Delta/\Omega$, defined by sending $e\in E_\Delta$ to the path component of $[e]_C$ in $E_\Omega$ and  $v\in V_\Delta$ to the path component of $v$ in $V_\Omega$.   For any edge $e\in E_\Delta$, $q(\opp{e})=\opp{q(e)}$ immediately from the definitions, while $\iota(q(e))=q(\iota(e))$ follows using the fact that $[e]_C$ and $\iota(e)$ are joined by an edge in $V_\Omega$; therefore, $q$ is indeed a morphism of graphs.
\end{definition}

The first example of an origami is the \emph{trivial} origami, which exists on any graph $\Delta$.

\begin{example}[Trivial origami]\label{ex: Trivial origami}
For any finite graph $\Delta$, the \emph{trivial} origami $\Omega$ is defined by taking $\sim_O$ to be equality on $E_\Delta$, meaning that $\sim_C$ is also equality. In this case, the edge graph $E_\Omega$ is just a disjoint union of intervals: each $e\in E_\Delta$ is incident at the vertices $[e]_O$ and $[e]_C$, each of valence 1. In the vertex graph $V_\Omega$, each vertex $v$ of $\Delta$ is incident at exactly those edges $e$ in the link $\Lk_\Delta(v)$; the terminus of such an edge $e$ is the singleton equivalence class $[e]_C$, which is not incident at any other edges. In summary, $V_\Omega$ is a disjoint union of `star graphs', one for each vertex of $\Delta$. Non-singularity, together with global and local consistency, are now trivial, so $\Omega$ is indeed an origami. Even more, since the vertex and edge graphs are forests, $\Omega$ is essential.
\end{example}

The next lemma explains the link between origamis and (un)folding.

\begin{lemma}[Unfolding origamis]\label{lem: Unfolding origamis}
Let $f=f_a:\Delta\to\Delta'$ be an essential fold morphism, and $\Omega'$ be an essential origami on the graph $\Delta'$. There is an essential origami $\Omega$ on $\Delta$ such that $f_a$ descends to an isomorphism $\Delta/\Omega\cong\Delta'/\Omega'$.
\end{lemma}
\begin{proof}
First, we fix some notation: let $u=\iota(a_1)=\iota(a_2)$, let $v_i=\tau(a_i)$, and for their images in $\Delta'$ we write $u'=f(u)$ and $v=f(v_1)=f(v_2)$. Apart from $u$, $v_1$ and  $v_2$, together with $a_1$, $a_2$ and their opposites, $f$ maps the vertices and edges of $\Delta$ bijectively to $\Delta'$.

We can now define the origami $\Omega$. To avoid a proliferation of notation, we shall use the notation $\sim_O$ for the open equivalence relations on both $\Delta$ and $\Delta'$ -- since they are distinct graphs, this should not cause confusion. The definition of $\sim_O$ on $\Delta$ falls into three cases.

\begin{enumerate}[(a)]
\item Unless $f(e_1)$ and $f(e_2)$ are open-equivalent to one of $a$ or $\opp{a}$, then $e_1\sim_O e_2$ if and only if $f(e_1)\sim_O f(e_2)$.
\item If $f(e)\sim_O a$ then $e\sim_{O}a_1\sim_Oa_2$.  In particular, $a_1\sim_O a_2$.
\item Finally, suppose that $f(e)\sim_O \opp{a}$. If $f(e)=\opp{a}$ then either $e=\opp{a}_1$ or $e=\opp{a}_2$, and we insist that $\opp{a}_1\nsim_O\opp{a}_2$. To complete the definition of $\Omega$ it remains to describe whether $e\sim_O\opp{a}_1$ or $e\sim_O\opp{a}_2$ when $f(e)\in [\opp{a}]_O\smallsetminus\{\opp{a}\}$.

To this end, let $\gamma$ be the (unique, since $V_{\Omega'}$ is a forest) embedded path in $V_{\Omega'}$ from $[f(e)]_C$ to $v$, and let $b$ be the edge of $\gamma$ adjacent to $v$. The local consistency of $\Omega'$ implies that $f(e)$ is the only edge of $[\opp{a}]_O$ that $\gamma$ traverses; in particular, $b\neq\opp{a}$ since $f(e)\neq \opp{a}$. By Remark \ref{rem: Unfolding and partitions}, the links of $v_1$ and $v_2$ induce a partition 
\[
\Lk_{\Delta'}(v)\smallsetminus\{\opp{a}\}=L_1\sqcup L_2
\]
where $L_i=f(\Lk_\Delta(v_i)\smallsetminus\{\opp{a}_i\})$.  Note that $\Lk_{V_{\Omega'}}(v')=\Lk_{\Delta'}(v')$ by construction, so $b\in \Lk_{\Delta'}(v)\smallsetminus\{\opp{a}\}$. We now set $e\sim_O \opp{a}_i$ for the unique $i=1,2$ such that $b\in L_i$.
\end{enumerate}

To check that $\Omega$ really is an origami with $\Delta/\Omega\cong \Delta'/\Omega'$, we first explain how the vertex and edge graphs of $\Omega$ relate to those for $\Omega'$.   The edge graph $E_\Omega$ is obtained from $E_{\Omega'}$ by dividing $a$ into $a_1$ and $a_2$, and likewise dividing $\opp{a}$ into $\opp{a}_1$ and $\opp{a}_2$. Since $a_1\sim_O a_2$ and $\opp{a}_1\sim_C \opp{a}_2$, this has the effect of unfolding $a$ and $\opp{a}$. In summary, $E_\Omega$ is obtained by unfolding $E_{\Omega'}$ twice. Furthermore, since $\opp{a}_1\nsim_O\opp{a}_2$ and $a_1\nsim_C a_2$, these unfolds are essential. This discussion is illustrated in Figure \ref{fig: Edge graph fold}. 

\begin{figure}[htp]
\begin{center}
\begin{tikzpicture}
\coordinate (aiO) at (-4,2);
\coordinate (a1c) at (-2,3);
\coordinate (a2c) at (-2,1);
\coordinate (a1) at (-3,2.5);
\coordinate (a2) at (-3,1.5);

\coordinate (aO) at (2,2);
\coordinate (ac) at (4,2);
\coordinate (a) at (3,2);

\coordinate (aiOlabel) at (-5.25,2);
\coordinate (a1clabel) at (-1.5,3);
\coordinate (a2clabel) at (-1.5,1);
\coordinate (a1label) at (-3,2.75);
\coordinate (a2label) at (-3,1.25);

\coordinate (aOlabel) at (1.5,2);
\coordinate (aclabel) at (4.5,2);
\coordinate (alabel) at (3,2.25);

\coordinate (EOmegalabel) at (-3,4);
\coordinate (EOmega'label) at (3,4);

\draw[thick,->] (aiO)-- (a1);\draw[thick,-] (a1) -- (a1c);
\draw[thick,->] (aiO) -- (a2);\draw[thick,-] (a2) -- (a2c);

\draw[thick,->] (aO) -- (a);\draw[thick,-] (a) -- (ac);

\coordinate (aibarc) at (-2,-2);
\coordinate (a1O) at (-4,-1);
\coordinate (a2O) at (-4,-3);
\coordinate (a1bar) at (-3,-1.5);
\coordinate (a2bar) at (-3,-2.5);

\coordinate (abarO) at (2,-2);
\coordinate (abarc) at (4,-2);
\coordinate (abar) at (3,-2);

\coordinate (aibarclabel) at (-1,-1.5);
\coordinate (a1Olabel) at (-4.5,-1);
\coordinate (a2Olabel) at (-4.5,-3);
\coordinate (a1barlabel) at (-3,-1.2);
\coordinate (a2barlabel) at (-3,-2.8);

\coordinate (abarOlabel) at (1.5,-2);
\coordinate (abarclabel) at (4.5,-2);
\coordinate (abarlabel) at (3,-1.7);

\draw[thick,->] (a1O)-- (a1bar);\draw[thick,-] (a1bar) -- (aibarc);
\draw[thick,->] (a2O) -- (a2bar);\draw[thick,-] (a2bar) -- (aibarc);

\draw[thick,->] (abarO) -- (abar);\draw[thick,-] (abar) -- (abarc);

\draw[->] (-1,2) -- (1,2);
\draw[->] (-1,-2) -- (1,-2);

\node at (EOmegalabel) {$E_\Omega$};
\node at (EOmega'label) {$E_{\Omega'}$};

\node at (aiOlabel) {$[a_1]_O=[a_2]_O$};
\node at (a1clabel) {$[a_1]_C$};
\node at (a2clabel) {$[a_2]_C$};
\node at (a1label) {$a_1$};
\node at (a2label) {$a_2$};
\node at (aOlabel) {$[a]_O$};
\node at (aclabel) {$[a]_C$};
\node at (alabel) {$a$};

\node at (aibarclabel) {$[\opp{a}_1]_C=[\opp{a}_2]_C$};
\node at (a1Olabel) {$[\opp{a}_1]_O$} ;
\node at (a2Olabel) {$[\opp{a}_2]_O$};
\node at (a1barlabel) {$\opp{a}_1$};
\node at (a2barlabel) {$\opp{a}_2$};

\node at (abarOlabel) {$[\opp{a}]_O$};
\node at (abarclabel) {$[\opp{a}]_C$};
\node at (abarlabel) {$\opp{a}$};
\end{tikzpicture}
\end{center}
\caption{Given an essential origami $\Omega'$ on $\Delta'$, an essential fold $\Delta\to\Delta'$ as in Figure \ref{fig: Fold} induces an origami $\Omega$ on $\Delta$ with an edge graph $E_\Omega$ obtained by unfolding $E_{\Omega'}$, as illustrated. These unfolds are always essential.}
\label{fig: Edge graph fold}
\end{figure}
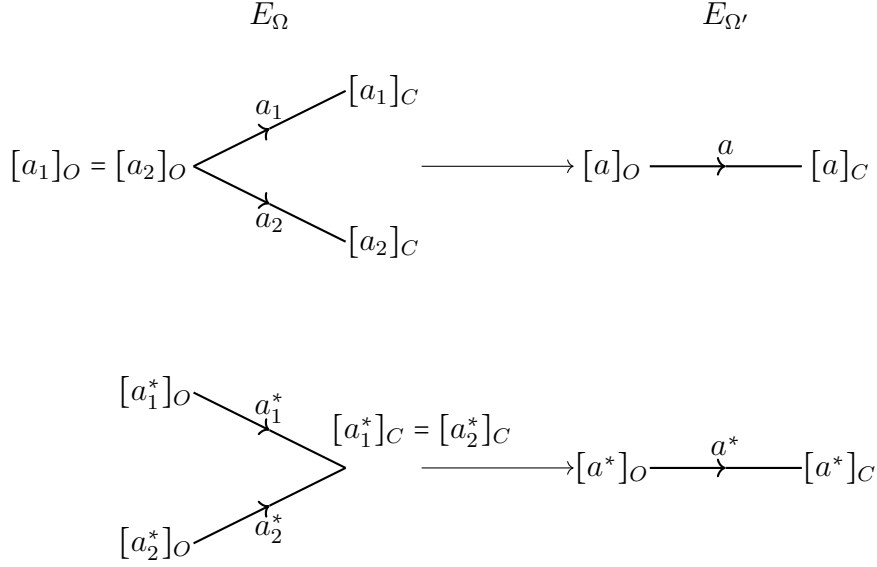

The vertex graph $V_\Omega$ is also obtained by unfolding $V_{\Omega'}$. The edge $a$ unfolds to the pair of edges $a_1$ and $a_2$, which meet at the vertex $u=\iota(a_1)=\iota(a_2)$. The terminal vertices $[a_i]_C$ are now distinct since $\opp{a}_1\nsim_O \opp{a}_2$, so this unfold is essential. Likewise, the edge $\opp{a}$ unfolds to the pair of edges $\opp{a}_1$ and $\opp{a}_2$, which meet at the vertex $[\opp{a}_1]_C=[\opp{a}_2]_C$, and their initial vertices $v_1$ and $v_2$ are distinct since $f$ is essential. See Figure \ref{fig: Vertex graph fold}.

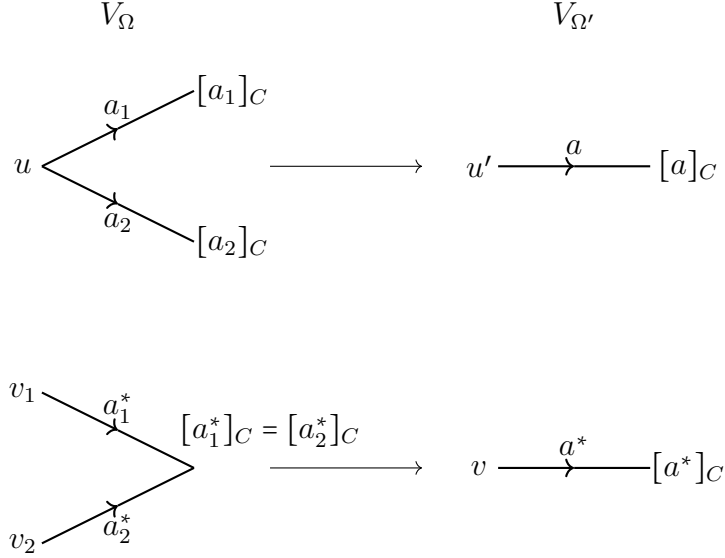
\begin{figure}[ht]
\begin{center}
\begin{tikzpicture}
\coordinate (u) at (-4,2);
\coordinate (a1c) at (-2,3);
\coordinate (a2c) at (-2,1);
\coordinate (a1) at (-3,2.5);
\coordinate (a2) at (-3,1.5);

\coordinate (u') at (2,2);
\coordinate (ac) at (4,2);
\coordinate (a) at (3,2);

\coordinate (ulabel) at (-4.25,2);
\coordinate (a1clabel) at (-1.5,3);
\coordinate (a2clabel) at (-1.5,1);
\coordinate (a1label) at (-3,2.75);
\coordinate (a2label) at (-3,1.25);

\coordinate (u'label) at (1.75,2);
\coordinate (aclabel) at (4.5,2);
\coordinate (alabel) at (3,2.25);

\coordinate (VOmegalabel) at (-3,4);
\coordinate (VOmega'label) at (3,4);

\draw[thick,->] (u)-- (a1);\draw[thick,-] (a1) -- (a1c);
\draw[thick,->] (u) -- (a2);\draw[thick,-] (a2) -- (a2c);

\draw[thick,->] (u') -- (a);\draw[thick,-] (a) -- (ac);

\coordinate (aibarc) at (-2,-2);
\coordinate (v1) at (-4,-1);
\coordinate (v2) at (-4,-3);
\coordinate (a1bar) at (-3,-1.5);
\coordinate (a2bar) at (-3,-2.5);

\coordinate (v) at (2,-2);
\coordinate (abarc) at (4,-2);
\coordinate (abar) at (3,-2);

\coordinate (aibarclabel) at (-1,-1.5);
\coordinate (v1label) at (-4.25,-1);
\coordinate (v2label) at (-4.25,-3);
\coordinate (a1barlabel) at (-3,-1.2);
\coordinate (a2barlabel) at (-3,-2.8);

\coordinate (vlabel) at (1.75,-2);
\coordinate (abarclabel) at (4.5,-2);
\coordinate (abarlabel) at (3,-1.7);

\draw[thick,->] (v1)-- (a1bar);\draw[thick,-] (a1bar) -- (aibarc);
\draw[thick,->] (v2) -- (a2bar);\draw[thick,-] (a2bar) -- (aibarc);

\draw[thick,->] (v) -- (abar);\draw[thick,-] (abar) -- (abarc);

\draw[->] (-1,2) -- (1,2);
\draw[->] (-1,-2) -- (1,-2);

\node at (VOmegalabel) {$V_\Omega$};
\node at (VOmega'label) {$V_{\Omega'}$};

\node at (ulabel) {$u$};
\node at (a1clabel) {$[a_1]_C$};
\node at (a2clabel) {$[a_2]_C$};
\node at (a1label) {$a_1$};
\node at (a2label) {$a_2$};
\node at (u'label) {$u'$};
\node at (aclabel) {$[a]_C$};
\node at (alabel) {$a$};

\node at (aibarclabel) {$[\opp{a}_1]_C=[\opp{a}_2]_C$};
\node at (v1label) {$v_1$} ;
\node at (v2label) {$v_2$};
\node at (a1barlabel) {$\opp{a}_1$};
\node at (a2barlabel) {$\opp{a}_2$};

\node at (vlabel) {$v$};
\node at (abarclabel) {$[\opp{a}]_C$};
\node at (abarlabel) {$\opp{a}$};

\end{tikzpicture}
\end{center}
\caption{Under an essential fold as in Figure \ref{fig: Fold}, the vertex graph $V_\Omega$ is obtained by unfolding $V_{\Omega'}$, as illustrated. These unfolds are always essential.}
\label{fig: Vertex graph fold}
\end{figure}

From these descriptions, it follows that $f$ preserves connectivity in the vertex and edge graphs: edges $e_1,e_2$ of $\Delta$ are in the same component of $E_\Omega$ if and only if $f(e_1)$ and $f(e_2)$ are in the same component of $E_{\Omega'}$, and the same is true for $V_\Omega$ and $V_{\Omega'}$. The same holds for vertices $v_1,v_2$: except in the trivial case where one is isolated, there are edges $e_i$ such that $v_i=\iota(e_i)$, and it follows from the case of edges that $v_1$ and $v_2$ are in the same component of $V_\Omega$ if and only if $f(v_1)$ and $f(v_2)$ are in the same component of $V_{\Omega'}$.

This connectivity information makes it possible to check that $\Omega$ is an origami. If $e$ and $\opp{e}$ were in the same component of $E_\Omega$ then $f(e)$ and $f(\opp{e})$ would be in the same component of $E_{\Omega'}$, a contradiction, which proves non-singularity.  Global consistency also follows quickly: by the definition of $\Omega$, if $e_1\sim e_2$ then $f(e_1)\sim_O f(e_2)$, so $\iota(f(e_1))$ is in the same component of $V_{\Omega'}$ as $\iota(f(e_2))$, by the global consistency of $\Omega'$. Since $f$ preserves connectivity, $\iota(e_1)$ and $\iota(e_2)$ are in the same component of $V_\Omega$, as required.

We next check local consistency. To this end, given edges $e_1, e_2\in [e]_O$, our goal is to prove that $e$ does not separate $\iota(e_1)$ from $\iota(e_2)$ in $V_
\Omega$. By the definition of $\Omega$, $f(e_1),f(e_2)\in [f(e)]_O$ so, by the local consistency of  $\Omega'$, the embedded path $\gamma'$ in $V_{\Delta'}$ from $\iota(f(e_1))$ to $\iota(f(e_2))$ doesn't cross $f(e)$. Since $V_\Delta$ is obtained from $V_{\Delta'}$ by unfolding the edges $a$ and $\opp{a}$, the path $\gamma'$ lifts to a path $\gamma$ in $V_\Omega$ unless $\gamma'$ crosses either of the vertices $v$ or $[a]_C$ of $V_{\Delta'}$. If it does lift, then $\gamma$ also doesn't cross $e$, as required.

Therefore, it remains to consider the cases in which $\gamma'$ fails to lift at either of the vertices $v$ or $[a]_C$ of $V_{\Omega'}$. Note that, since $\gamma'$ is an embedded path, it crosses each of these vertices at most once.

Suppose first that $\gamma'$ crosses $[a]_C$ but not $v$.  For $i=1,2$, let $\gamma'_i$ be the embedded path from $\iota(f(e_i))$ to $[a]_C$, so $\gamma'$ is the concatenation $\gamma'_1\cdot\gamma'_2$. Each $\gamma'_i$ lifts to a path $\gamma_i$ in $V_\Omega$ from $\iota(e_i)$ to $[a_{j(i)}]_C$ for some $j(i)=1,2$. Therefore, the concatenation
\[
\gamma=\gamma_1\cdot a_{j(1)}\cdot a_{j(2)}\cdot\gamma_2
\]
is a path in $V_\Omega$ from $\iota(e_1)$ to $\iota(e_2)$. Neither $\gamma_1$ nor $\gamma_2$ cross $e$, since $\gamma'$ doesn't cross $f(e)$, so it remains to show that $e$ is not equal to $a_1$ or $a_2$. But $a$ separates $u=\iota(a)$ from at least one of $\iota(f(e_1))$ and $\iota(f(e_2))$, so $a\nsim_O f(e_i)$ and therefore $a_j\nsim_O e_i$, whence $e$ is neither $a_1$ nor $a_2$, as required.

Next, suppose that $\gamma'$ crosses $v$ but not $[a]_C$, and let $\gamma'_i$ be the embedded path from $\iota(f(e_i))$ to $v$. Again, each $\gamma'_i$ lifts to a path $\gamma_i$ in $V_\Omega$ from $\iota(e_i)$ to some $v_{k(i)}$. If $f(e_i)\nsim_O\opp{a}$ then the concatenation
\[
\gamma=\gamma_1\cdot \opp{a}_{k(1)}\cdot \opp{a}_{k(2)}\cdot\gamma_2
\]
does not cross $e$, as in the previous case. Hence, we may assume that $f(e_i)\sim_O\opp{a}$. Let $b'_i$ be the edge of $\gamma'_i$ incident at $v$ (or take $b'_i=f(e_i)$ if $\gamma'_i$ is just a point). Local consistency of $\Omega'$ implies that $\gamma'$ does not cross $\opp{a}$. Therefore, if $b'_i=\opp{a}$, it follows that $\gamma'_i$ is a point and $f(e_i)=\opp{a}$, so $e_i=\opp{a}_{k(i)}$. On the other hand, if $b'_i\neq\opp{a}$, then $b'_i\in L_{k(i)}$, as in item (c) of the definition of $\Omega$.  If $k(1)\neq k(2)$ then item (c) of the definition of $\Omega$ implies that $e_1\nsim_O e_2$, which contradicts our hypothesis that $e_1,e_2\in[e]_O$. Therefore, $k(1)=k(2)$, so the lifts $\gamma_1$ and $\gamma_2$ both end at the same vertex $v_k$, and the concatenation $\gamma=\gamma_1\cdot\gamma_2$ is a path from $\iota(e_1)$ to $\iota(e_2)$ that does not traverse $e$, as required.

It remains to deal with the case when $\gamma'$ crosses both $v$ and $[a]_C$. In this case, swapping $e_1$ and $e_2$ if necessary, we may write
\[
\gamma'=\gamma'_1\cdot\delta'\cdot\gamma'_2
\]
where $\gamma'_1$ is the shortest path from $\iota(f(e_1))$ to $[a]_C$, $\delta'$ is the shortest path from $[a]_C$ to $v$ and $\gamma'_2$ is the shortest path from $v$ to $\iota(f(e_2))$. The arguments of the previous two cases now apply verbatim to prove that $e$ does not separate $\iota(e_1)$ from $\iota(e_2)$. In more detail, suppose that $\gamma_i$ is the lift of $\gamma'_i$ to $V_\Omega$ and $\delta$ is the lift of $\delta'$. As above, $e$ is not equal to $a_1$ or $a_2$ so a concatenation
\[
\gamma=\gamma_1\cdot a_{j(1)}\cdot a_{j(2)}\cdot\delta\cdot \opp{a}_{k(1)}\cdot\opp{a}_{k(2)}\cdot\gamma_2
\]
is the required path that does not cross $e$ unless $e=\opp{a}_1$ or $\opp{a}_2$.  If $e=\opp{a}_k$ for some $k$ then, as above, the hypothesis that $e_1\sim_O e_2$ and item (c) of the definition of $\Omega$ imply that $\delta$ and $\gamma_2$ end at the same vertex $v_k$, so
\[
\gamma=\gamma_1\cdot a_{j(1)}\cdot a_{j(2)}\cdot\delta\cdot\gamma_2
\]
is the required path.

Finally, that $\Delta/\Omega\cong\Delta/\Omega'$ is also an immediate consequence of the fact that $f$ preserves connectivity in the vertex and edge graphs.
\end{proof}

Since homotopy equivalences factor as compositions of essential folds, it follows that they can be realised by origamis. 

\begin{proposition}\label{prop: Essential origamis realise homotopy equivalences}
Let $f:\Delta\to\Gamma$ be a morphism of finite, connected, core graphs.  If $f$ is a homotopy equivalence, then there is an essential origami $\Omega$ on $\Delta$ and an isomorphism $\Delta/\Omega\cong\Gamma$ that identifies $f$ with the quotient map.
\end{proposition}
\begin{proof}
By Lemma \ref{lem: Stallings folding lemma},  $f$ factors as
\[
\Delta\stackrel{f_0}{\to}\overline{\Delta}\stackrel{\bar{f}}{\to}\Gamma\,,
\]
where $f_0$ is a composition of folds and $\bar{f}$ is an immersion. Since $f$ is itself $\pi_1$-surjective, if follows that $\bar{f}$ is also $\pi_1$-surjective and hence injective by a standard argument\footnote{If $\bar{f}(u)=\bar{f}(v)$ with $u$ and $v$ distinct vertices then a shortest path $\gamma$ from $u$ to $v$ maps to a loop in $\Gamma$ representing an element of $\pi_1(\Gamma)$ not in $f_*\pi_1(\Delta)$.}. Since $\Gamma$ is core and $\bar{f}$ is a homotopy equivalence, it follows that $\bar{f}$ is an isomorphism.

Thus, it suffices to identify $\overline{\Delta}$ with a quotient $\Delta/\Omega$, in such a way that $f_0$ becomes the quotient map. This is done by induction on $n$, where  $f_0$ factors as a composition of $n$ folds. In the base case $n=0$, the trivial origami of Example \ref{ex: Trivial origami} is as required. The inductive step is provided by Lemma \ref{lem: Unfolding origamis}, noting that every fold is essential since $f_0$ is injective on fundamental groups. This completes the proof that $f$ is realised by an essential origami.
\end{proof}

The most important theorem about origamis is a kind of converse to Proposition \ref{prop: Essential origamis realise homotopy equivalences}: if an origami is essential then the quotient map is a $\pi_1$-injection. This is how origamis make it possible to certify $\pi_1$-injectivity, even for morphisms of graphs that are not immersions. The next lemma provides the key step in the proof.

\begin{lemma}\label{lem: Producing a foldable pair}
Let $\Omega$ be an essential origami on a finite graph $\Delta$. If the quotient map $q:\Delta\to\Delta/\Omega$ is not an immersion then there are distinct edges $a_1$ and $a_2$ such that $\iota(a_1)=\iota(a_2)$ and $a_1\sim_O a_2$.
\end{lemma}
\begin{proof}
Since the quotient map $q$ is not an immersion, it folds some pair of edges by Stallings' folding lemma. That is, there are distinct edges $a_1,a_2$ of $\Delta$ with $u=\iota(a_1)=\iota(a_2)$ and with $q(a_1)=q(a_2)$, which means that $a_1$ and $a_2$ are in the same component of $E_\Omega$. We will prove that $a_1\sim_O a_2$.

Because $E_\Omega$ is a forest, there is a unique shortest edge-path
\[
e_1,e_2,\ldots,e_{2n-1},e_{2n}
\]
in $E_\Omega$ from $[a_1]_O$ to $[a_2]_O$. That is, the edges $e_i$ satisfy
\[
a_1\sim_O e_1\sim_C e_2\sim_O \ldots \sim_O e_{2n-1}\sim_C e_{2n}\sim_O a_2
\]
and $n$ is minimal.

We now use this path in $E_\Omega$ to construct a path $\beta$ in $V_\Omega$ that starts and ends at $u$. Let $\alpha_0$ be the shortest path in $V_\Omega$ from $u=\iota(a_1)$ to $\iota(e_1)$; for each $1\leq i< n$, let $\alpha_i$ be the shortest path from $\iota(e_{2i})$ to $\iota(e_{2i+1})$; and finally, let $\alpha_n$ be the shortest path from $\iota(e_{2n})$ to $\iota(a_2)=u$. Note that any of the paths $\alpha_i$ might be points, and indeed this will certainly occur if $a_1=e_1$ or $a_2=e_{2n}$. The required path is the concatenation
\[
\beta= \alpha_0\cdot e_1\cdot e_2 \cdot \alpha_1 \cdot \cdots \cdot \alpha_{n-1}\cdot e_{2n-1} \cdot e_{2n}\cdot \alpha_n\,.
\]
Note that $\beta$ is an immersed path in $V_\Omega$.  Indeed, $e_{2i-1}$ and $e_{2i}$ are distinct by the minimality of $n$, while $\alpha_i$ does not traverse $e_{2i}$ or $e_{2i+1}$, by the local consistency of $\Omega$.

In summary, $\beta$ is an immersed path in the forest $V_\Omega$, starting and ending at $u$. Since any immersed path in a tree must be embedded, $\beta$ has length zero, which in turn implies that $n=0$ and $a_1\sim_O a_2$, as required. 
\end{proof}

Lemma \ref{lem: Producing a foldable pair} is useful because it enables us to fold $\Delta$. The next lemma shows how to push the essential origami $\Omega$ down to an essential origami on the folded graph.

\begin{lemma}[Folding origamis]\label{lem: Folding origamis}
Let $\Omega$ be an essential origami on a finite graph $\Delta$. Consider two edges $a_1,a_2$ of $\Delta$ with $\iota(a_1)=\iota(a_2)$, and let $f=f_a:\Delta\to\Delta'$ be the associated fold. If $a_1\sim_O a_2$ then there is an essential origami $\Omega'$ on $\Delta'$ such that $f$ descends to an isomorphism $\Delta/\Omega\cong\Delta'/\Omega'$. Furthermore, the fold $f$ is also essential.
\end{lemma}
\begin{proof}
We will use the same notation as in the proof of Lemma \ref{lem: Unfolding origamis}: $u=\iota(a_1)=\iota(a_2)$, $v_i=\tau(a_i)$ and for their images in $\Delta'$ we write $u'=f(u)$ and $v=f(v_1)=f(v_2)$. As in the earlier lemma, apart from $u$, $v_1$ and  $v_2$, together with $a_1$, $a_2$ and their opposites, $f$ maps the vertices and edges of $\Delta$ bijectively to $\Delta'$.

To define the origami $\Omega'$ we need to specify the open equivalence relation $\sim_O$ on the edges of $\Delta'$, and we do this by setting it to be the finest equivalence relation on the edges of $\Delta'$ with the property that $e_1\sim_O e_2$ implies that $f(e_1)\sim_O f(e_2)$. This relation can be described exactly using a brief case analysis: if neither $e_1$ nor $e_2$ is open-equivalent to any of $a_1$, $a_2$, $\opp{a}_1$ or $\opp{a}_2$, then $f(e_1)\sim_O f(e_2)$ if and only if $e_1\sim_O e_2$; if $e_1\sim_O a_i$ then $f(e_2)\sim_O f(e_1)$ if and only if $e_2\sim_O a_1$ or $e_2\sim_O a_2$; and if $e_1\sim_O \opp{a}_i$ then $f(e_2)\sim_O f(e_1)$ if and only if $e_2\sim_O \opp{a}_1$ or $e_2\sim_O \opp{a}_2$.

To check that this defines an origami $\Omega'$ on $\Delta'$, we need to describe the vertex graph $V_{\Omega'}$ and the edge graph $E_{\Omega'}$. These are both obtained by the reverse procedure to that of the proof of Lemma \ref{lem: Unfolding origamis}: that is, they are obtained by folding the pairs $\{a_1,a_2\}$ and $\{\opp{a}_1,\opp{a}_2\}$ in $V_\Omega$ and $E_\Omega$, respectively.  In particular, connectivity in the vertex and edge graphs is preserved, and the fact that $\Omega'$ is an essential origami with the desired properties follows quickly, except for the claim that $\Omega'$ is locally consistent, which is more delicate. We will therefore check local consistency in detail, and leave the reader to fill in the remaining arguments, which are straightforward.

Suppose, therefore, that $e'_1,e'_2\in [e']_O$ for edges $e'_i=f(e_i)$ and $e'=f(e)$. To check local consistency, we need to find a path in $V_{\Omega'}$ from $\iota(e'_1)$ to $\iota(e'_2)$ that does not cross $e'$.  By the definition of $\Omega'$, unless $e'\sim_O\opp{a}$, we have that $e_1,e_2\in[e]_O$. Therefore, by the local consistency of $\Omega$, the shortest path $\gamma$ in $V_\Omega$ from $\iota(e_1)$ to $\iota(e_2)$ does not cross $e$, and so $f(\gamma)$ is the required path from $\iota(e'_1)$ to $\iota(e'_2)$.

It remains to consider the case in which $e'\sim_O\opp{a}$. By the definition of $\Omega'$, $e_1\sim_O \opp{a}_i$ and $e_2\sim_O\opp{a}_j$, for some $i,j\in\{1,2\}$. We claim that $\gamma_1$, the shortest path in $V_\Omega$ from $\iota(e_1)$ to $v_i$, crosses no preimage of $e'$. Indeed, let $\epsilon$ be any edge crossed by $\gamma_1$. The local consistency of $\Omega$ implies that $\epsilon\nsim_O\opp{a}_i$. Also, some subpath $\beta$ of $\gamma_1$ is an embedded path from $\iota(\epsilon)$ to $v_i$, so the concatenation of $\beta$ with $\opp{a}_1$ and $\opp{a}_2$ (suitably oriented) is the unique embedded path from $\iota(\epsilon)$ to $v_{3-i}$. Since this embedded path traverses $\opp{a}_{3-i}$, it follows that $\epsilon\nsim_O \opp{a}_{3-i}$ by the local consistency of $\Omega$. Thus $\epsilon$ is open-equivalent to neither $\opp{a}_1$ nor $\opp{a}_2$, so $f(\epsilon)\nsim_O \opp{a}$, and in particular $f(\epsilon)\neq e'$, as required.

Similarly, the shortest path $\gamma_2$ from $\iota(e_2)$ to $v_j$ also crosses no preimage of $e'$. The images $f(\gamma_1)$ and $f(\gamma_2)$ both end at $v$, and so their concatenation is a path in $V_{\Omega'}$ from $\iota(e'_1)$ to $\iota(e'_2)$ that does not cross $e'$. This proves the local consistency of $\Omega'$, which completes the proof.
\end{proof}

With these lemmas in hand, we can now prove the aforementioned theorem: essential origamis induce $\pi_1$-injective maps.

\begin{theorem}\label{thm: Essential origamis give injective maps of groups}
Let $\Omega$ be an origami on a finite, connected graph $\Delta$.  If $\Omega$ is essential then the quotient map $q:\Delta\to\Delta/\Omega$ induces an injective homomorphism of fundamental groups.
\end{theorem}
\begin{proof}
The proof is by induction on the difference between the number of edges of $\Delta$ and the number of edges of $\Delta/\Omega$. In the base case, that difference is zero, meaning that every component of $E_\Omega$ contains a unique edge of $E_\Delta$. This implies that $\Omega$ is the trivial origami of Example \ref{ex: Trivial origami}, so the quotient map is an isomorphism and the result is immediate.

For the inductive step, either $q$ is an immersion and hence $\pi_1$-injective, so there is nothing to prove, or Lemma \ref{lem: Producing a foldable pair} provides a pair of edges $a_1$ and $a_2$. By Lemma \ref{lem: Folding origamis}, the quotient map $q$ factors through the essential fold $f=f_a:\Delta\to \Delta'$ that identifies $a_1$ and $a_2$, and $\Omega$ descends to an essential origami $\Omega'$ on $\Delta'$.

The folded graph $\Delta'$ has one fewer edges than $\Delta$, so the quotient map $q':\Delta'\to\Delta'/\Omega'$ is $\pi_1$-injective by induction. Thus
\[
q=q'\circ f
\]
is also $\pi_1$-injective, as required, because $f$ is an essential fold.
\end{proof}

The theorem tells us that essential origamis can be used to certify the $\pi_1$-injectivity of a quotient map. More generally, origamis can also provide certificates of $\pi_1$-injectivity for morphisms that factor through their quotient maps.

\begin{definition}\label{def: Compatible origami}
Let $f:\Delta\to\Gamma$ be a morphism of graphs. An origami $\Omega$ on $\Delta$ is said to be \emph{compatible with $f$} if $f$ factors through the quotient map $\Delta\to\Delta/\Omega$, and the factor map $\Delta/\Omega\to\Gamma$ is an immersion.  
\end{definition}

\begin{remark}\label{rem: Certifying compatibility}
Note that compatibility is easy to certify: $f$ is compatible with $\Omega$ if and only if all of the following conditions hold.
\begin{enumerate}[(i)]
\item If edges $e_1,e_2$ of $\Delta$ are in the same component of $E_\Omega$ then $f(e_1)=f(e_2)$.
\item If vertices $v_1,v_2$ of $\Delta$ are in the same component of $V_\Omega$ then $f(v_1)=f(v_2)$.
\item If $e_1,e_2$ are edges of $\Delta$ with $f(e_1)=f(e_2)$ and $\iota(e_1),\iota(e_2)$ are in the same component of $V_\Omega$ then $e_1,e_2$ are in the same component of $E_\Omega$.
\end{enumerate}
\end{remark}

Putting together the results of this section, we can prove that origamis can be used to certify $\pi_1$-injectivity for morphisms of graphs.

\begin{proof}[Proof of Theorem \ref{thm: Essential origamis}]
If there is an essential origami $\Omega$ compatible with $f$ then, by definition, $f$ factors as $\bar{f}\circ q$, where $\bar{f}$ is an immersion and $q$ is the quotient map $\Delta\to\Delta/\Omega$. By Theorem \ref{thm: Essential origamis give injective maps of groups}, $q$ is $\pi_1$-injective because $\Omega$ is essential, while $\bar{f}$ is $\pi_1$-injective because it is an immersion, so $f$ is also $\pi_1$-injective. 

For the other direction, by Stallings' folding lemma $f$ factors as $\bar{f}\circ f_0$, where $\bar{f}$ is an immersion and $f_0$ is $\pi_1$-surjective.  Since $f$ is $\pi_1$-injective, it follows that $f_0$ is in fact a homotopy equivalence and so, by Proposition \ref{prop: Essential origamis realise homotopy equivalences}, $\Delta$ admits an essential origami  for which $f_0$ is the quotient map. Thus, $\Omega$ is indeed compatible with $f$, as required.
\end{proof}

\section{Branched 2-complexes and their morphisms}\label{sec: Essential maps of 2-complexes}

Intuitively, a 2-dimensional cell complex, or 2-complex, is the result of gluing a set of discs to a graph along their boundaries.  This is a fundamental notion in topology, but our setting requires a few unusual modifications of the standard definition, and the precise combinatorial model that we use will also be important in proving the rationality theorem. We therefore develop the definitions carefully here.

\subsection{Branched 2-complexes}

\begin{definition}[2-complex]\label{def: 2-complex}
A \emph{2-complex} $X$ consists of a pair of (Serre) graphs $\Gamma=\Gamma_X$ and $S=S_X$, together with an \emph{attaching map} $w=w_X:S_X\to\Gamma_X$ subject to two conditions:
\begin{enumerate}[(i)]
\item the realisation of $S$ is a disjoint union of circles; and
\item the morphism $w$ is required to be an immersion.
\end{enumerate}
The graph $\Gamma$ is the \emph{1-skeleton} of $X$.  The \emph{vertices} of $X$ are the vertices of $\Gamma$, so we write $V_X$ for $V_\Gamma$, and likewise the \emph{edges} of $X$ are the edges of $\Gamma$, so we write $E_X$ for $E_\Gamma$. The path components $f\in\pi_0S$ are called the \emph{2-cells} or \emph{faces} of $X$, so we write $F_X$ for $\pi_0S$.  The \emph{realisation} of $X$, also denoted by $X$, is obtained by coning off each face $f$ to produce a disc $D_f$, and then gluing the discs $D_f$ to the 1-skeleton $\Gamma$ along the map $w$. Write $D=D_X$ for the disjoint union of the discs, so the realisation is $\Gamma\cup_w D$.

A \emph{morphism} of 2-complexes $\phi:Y\to X$ consists of a morphism of graphs $\phi:\Gamma_Y\to\Gamma_X$ and a morphism of graphs $\phi:S_Y\to S_X$, injective on path components, such that $\phi\circ w_Y=w_X\circ \phi$. Note that a morphism induces a continuous map of realisations (also denoted by $\phi:Y\to X$) which sends cells of $Y$ homeomorphically to cells of $X$.
\end{definition}

\begin{remark}
The second hypothesis, that $w_X$ should be an immersion, differentiates this definition from the standard one. However, note that it can always be achieved by modifying $w_X$ by a homotopy, except in the trivial case where some face is attached along a loop that is null-homotopic in the 1-skeleton.  After discarding such trivial faces, there is therefore no harm in imposing the second hypothesis.
\end{remark}

For a wealth of finite examples, the reader should have in mind the standard presentation complex of any finite group presentation
\[
G=\langle a_1,\ldots,a_m\mid w_1,\ldots,w_n\rangle\,.
\]
The second hypothesis in the definition corresponds to assuming that each relator $w_j$ is a non-trivial element of the free group on the generators.

In order to achieve our goal of `linearising' the study of 2-complexes, it is fruitful to enrich them by equipping them with some extra information -- a notion of area. This leads to the definition of a \emph{branched} 2-complex. A branched 2-complex is just a 2-complex with a little extra structure, but some care is needed about the morphisms allowed between them.

\begin{definition}[Branched 2-complex]\label{def: Branched 2-complex}
A \emph{branched 2-complex} is a 2-complex $X$ and an \emph{area function} $\Area:F_X\to\mathbb{R}_{\geq 0}$. That is, the area function assigns a non-negative area $\Area(f)$ to each face $f$ of $X$. 

Morphisms of branched complexes are called \emph{branched morphisms}, to distinguish them from the standard morphisms of 2-complexes in Definition \ref{def: 2-complex}, and are defined as follows.

A \emph{branched morphism} $\phi:Y\to X$ consists of a morphism of graphs $\phi:\Gamma_Y\to\Gamma_X$ and an immersion $\phi: S_Y\to S_X$ such that $\phi\circ w_Y=w_X\circ \phi$, and also subject to the following condition. For each face $f$ of $Y$, the immersion $\phi$ necessarily defines a covering map of circles $\phi_f:f\to \phi(f)$. The degree of this covering map is denoted by $m_\phi(f)$ and is called the \emph{multiplicity} of $f$. The extra condition requires that
\[
\Area(f)=m_\phi(f)\Area(\phi(f))
\]
for each face $f$ of $Y$.

As in the case of unbranched 2-complexes, a branched morphism induces a continuous map of realisations. However, while this continuous map sends vertices to vertices and edges homeomorphically to edges, the induced maps between faces may no longer be homeomorphic. Rather, on the realisations of faces $D_f\to D_{\phi(f)}$, the realisation of $\phi$ can be modelled by the polynomial map $z\mapsto z^{m_\phi(f)}$ on the unit disc in the complex plane. This justifies the terminology `branched complex' and `branched morphism'.
\end{definition}

Since branched 2-complexes are in particular 2-complexes, many definitions can be transported directly from 2-complexes to branched 2-complexes. For instance, the fundamental group of a branched 2-complex is the fundamental group of its underlying 2-complex. However,  branched 2-complexes can also be thought of as a generalisation of 2-complexes, since we allow more morphisms between them. 

\begin{remark}\label{rem: Standard 2-complexes}
A 2-complex $X$ admits a canonical structure as a branched complex, by setting $\Area(f)=1$ for each face $f$ of $X$.  Such branched 2-complexes will sometimes be called \emph{standard}. Note that branched morphisms between standard 2-complexes are exactly the usual morphisms specified in Definition \ref{def: 2-complex}.
\end{remark}

Each vertex $v$ of a (branched) 2-complex $X$ has a \emph{link} $\Lk_X(v)$, a naturally defined graph whose realisation can be thought of as a small sphere around $v$. It can also be defined formally as follows.

\begin{definition}[Links in 2-complexes]\label{def: Links in 2-complexes}
Let $X$ be a 2-complex defined by an immersion of graphs $w:S\to \Gamma$ as above.  Both vertices and edges of $X$ have links. The \emph{link of an edge} $e$ of $X$ is the set
\[
\Lk_X(e)=w^{-1}(e)\,.
\]
The \emph{link of a vertex} $v$ of $X$ is the graph $L=\Lk_X(v)$, defined as follows. The vertex set of $L$ is $\Lk_\Gamma(v)$, i.e.\ the set of edges of $\Gamma$ with initial vertex $v$, and the edge set of $L$ is exactly $w^{-1}(\Lk_\Gamma(v))$, i.e.\ the edges of $S$ whose initial vertex maps to $v$. The terminus map $\tau_L$ of $L$ is defined to coincide with $w$ (so the terminal vertex of $e$ is $w(e)$) and so, to complete the definition of $L$, it remains to specify the orientation-reversing involution on edges, which is defined as follows: for an edge $e$, the opposite edge $\opp{e}$ is the unique edge of $S$ not equal to $e$ such that $\iota_S(\opp{e})=\iota_S(e)$. (Note that $\opp{e}$ exists and is unique precisely because $S$ is a disjoint union of circles.)

For each vertex $v$ of $X$, a branched morphism $\phi:X\to Y$ induces a morphism of graphs $d\phi_v:\Lk_X(v)\to\Lk_Y(\phi(v))$.
\end{definition} 

\begin{remark}\label{rem: No self-loops in links}
The requirement in Definition \ref{def: 2-complex} that $w$ should be an immersion means that it does not factor through a fold, so
\[
\tau_L(e)=w(e)\neq w(\opp{e})=\tau_L(\opp{e})\,.
\]
Hence, the two endpoints of each edge in $\Lk_X(v)$ are distinct.
\end{remark}

Links of vertices and edges come with some important extra structure.

\begin{remark}\label{rem: Extra structure on links}
The orientation-reversing involution on the edges of $S$ induces a canonical bijection $\Lk_X(e)\to \Lk_X(\opp{e})$, for any edge $e$ of $X$. Likewise, the origin map $\iota$ canonically identifies $\Lk_X(e)$ with a set of edges in $\Lk_X(\iota(e))$: indeed, they are exactly the edges of $\Lk_X(\iota(e))$ that adjoin (the vertex) $e$ in $\Lk_X(\iota(e))$. Formally, this says that $\iota$ induces a bijection 
\[
\Lk_X(e)\to \Lk_{\Lk_X(\iota(e))}(e)
\]
where, confusingly, $e$ is thought of as a \emph{vertex} of the graph  $\Lk_X(\iota(e))$.  Equipped with the extra data of the orientation-reversing involutions on links of edges and the inclusions of links of edges into links of vertices, the entire complex $X$ can be reconstructed from the set of links of edges and vertices.
\end{remark}

\subsection{Essential maps}

Just as Stallings' folding lemma explains how to put maps of graphs into a useful form, so we may likewise fold maps of branched 2-complexes.  Here, the role of immersions is played by \emph{branched immersions}.

\begin{definition}\label{def: Branched immersion}
A branched morphism $\phi:Y\to X$ is a \emph{branched immersion} if its realisation is locally injective away from the centres of 2-cells. Equivalently, the induced map on links $d\phi_v$ is injective for each vertex $v$ of $Y$. In the case where $X$ and $Y$ are standard, such a map is just called an \emph{immersion}, and the realisation is locally injective everywhere.
\end{definition}

The next result is the analogue of Stallings' folding lemma (Lemma \ref{lem: Stallings folding lemma}) for 2-complexes. Cousins of this result have already been used extensively (see \cite{louder_stackings_2017,louder_one-relator_2020,louder_negative_2022,louder_uniform_2024}), but we give a careful proof here for completeness. 

\begin{lemma}[Folding 2-complexes]\label{lem: Folding 2-complexes}
Every branched morphism of finite, connected, branched 2-complexes  $\phi:Y\to X$ factors as
\[
Y\stackrel{\phi_0}{\to}\bar{Y}\stackrel{\bar{\phi}}{\to} X
\]
where $\phi_0$ is a surjection on fundamental groups and $\bar{\phi}$ is a branched immersion.  Furthermore, $\bar{Y}$ enjoys a universal property: whenever $\phi$ factors through a branched immersion $\psi:Z\to X$, $\bar{\phi}$ also uniquely factors through $Z$.
\end{lemma}

The idea of the proof of Lemma \ref{lem: Folding 2-complexes} is straightforward: fold the map of 1-skeleta using Lemma \ref{lem: Stallings folding lemma}, and then identify any 2-cells whose images coincide. The only subtlety is to make precise sense of the idea of 2-cells coinciding, especially in the broader setting of branched complexes. This can be done using fibre products of graphs, as described by Stallings \cite[\S1.3]{stallings_topology_1983}. %The next remark relates branched immersions of 2-complexes to fibre products.

A branched morphism of 2-complexes $\phi:Y\to X$ can be thought of as a commutative diagram, as follows.
\begin{center}
  \begin{tikzcd}
S_Y\arrow{r}{\phi}\arrow{d}{w_Y} & S_X\arrow{d}{w_X}\\
\Gamma_Y\arrow{r}{\phi} & \Gamma_X
  \end{tikzcd}
\end{center}
Recall that we always assume that $w_X$ and $w_Y$ are immersions of graphs.

\begin{lemma}\label{lem: Immersions and fibre products}
A branched morphism $\phi:Y\to X$ is a branched immersion if and only if:
\begin{enumerate}[(i)]
\item the map of 1-skeleta $\phi:\Gamma_Y\to \Gamma_X$ is an immersion of graphs; and
\item the canonical map to the fibre product
\[
S_Y\to \Gamma_Y\times_{\Gamma_X} S_X
\]
is an embedding.
\end{enumerate}
\end{lemma}
\begin{proof}
Throughout the proof, it is useful to recall that, on each 2-cell $D_f$ of $Y$, the realisation of $\phi$ can be parametrised as the polynomial
\[
z\mapsto z^{m_\phi(f)}\,,
\]
where both $D_f$ and $D_{\phi(f)}$ are identified with the unit disc in the complex plane. In particular, each $s\in S_Y$ and $r\in S_X$ is identified with a complex number of modulus 1. We will use this notation for maps between 2-cells.

If $\phi$ is a branched immersion, then its restriction to the complement of the centres of the 2-cells of $Y$ is locally injective. In particular, the restriction to the 1-skeleton $\phi:\Gamma_Y\to\Gamma_X$ is an immersion, which proves (i).

For (ii), consider distinct $s_1,s_2\in S_Y$ having the same image $(y,r)$ in the fibre product. Let $s_i$ be contained in the face $f_i$ of $Y$, and let $m_i=m_{\phi}(f_i)$, for each $i=1,2$.  If the realisation map $\phi$ is a branched immersion then it is injective in some neighbourhood $U$ of $y\in\Gamma_Y\subseteq Y$.  Since $w_Y(s_1)=w_Y(s_2)=y$, in the realisation of $Y$ the points $s_1$ and $s_2$ in $S_Y$ are both identified with $y\in\Gamma_Y$.   Therefore, we may choose $t<1$ large enough that both $t^{1/m_1}s_1$ and $t^{1/m_2}s_2$ are contained in $U$ (thinking of $s_1$ and $s_2$ as unit-modulus complex numbers). By definition of the realisation map,
\[
\phi(t^{1/m_1}s_1)=(t^{1/m_1})^{m_1}\phi(s_1)=tr=(t^{1/m_2})^{m_2}\phi(s_2)=\phi(t^{1/m_2}s_2)
\]
so, since $\phi$ is injective on $U$, $t^{1/m_1}s_1=t^{1/m_2}s_2$. In particular, by comparing arguments, $s_1=s_2$ as required.

To prove the converse direction, note that all branched morphisms are locally injective away from the the centres of the 2-cells and the 1-skeleton, so it suffices to prove local injectivity around a point $y\in \Gamma_Y$. By item (i), we may choose an open neighbourhood $W$ of $y$ in the 1-skeleton $\Gamma_Y$ such that $\phi$ is injective on $W$. Let $D_Y^0$ denote the complement of the centres  in $D_Y$, the disjoint union of the 2-cells. Note that the map $w_Y:S_Y\to \Gamma_Y$ extends continuously to $D^0_Y$ by setting $w_Y(ts)=w_Y(s)$.  Let
\[
U= W\cup w_Y^{-1}(W)\,,
\]
an open neighbourhood of $y$ in the realisation of $Y$. It suffices to show that $\phi$ is injective on $U$. Injectivity is immediate on $W$.

Suppose that $t_1s_1$ and $t_2s_2\in U\smallsetminus W$ have the same image under $\phi$. As above, let $s_i$ be contained in the face $f_i$ of $Y$, and let $m_i=m_{\phi}(f_i)$, for each $i=1,2$, so $\phi(t_is_i)=t_i^{m_i}\phi(s_i)$. In particular, by examining the argument we see that $\phi_1(s_1)$ and $\phi_2(s_2)$ have the same image $r\in S_X$, while by examining the modulus we see that $t_1^{m_1}=t_2^{m_2}$. Commutativity of the associated diagram implies that
\[
\phi\circ w_Y(s_i)=w_X\circ\phi(s_i)=w_X(r)
\]
for each $i=1,2$ so, since $\phi$ is injective on $W$, $w_Y(s_1)=w_Y(s_2)=y$ for some $y\in W\subseteq  \Gamma_Y$. Now, $s_1$ and $s_2$ both map to $(y,r)$ in the fibre product so, by item (ii), $s_1=s_2$. In particular, they lie in the same face, so $m_1=m_2$, whence $t_1=t_2$,  which completes the proof.
\end{proof}

With this lemma in hand, the proof of the folding lemma becomes routine.

\begin{proof}[Proof of Lemma \ref{lem: Folding 2-complexes}]
Applying Stallings' folding lemma to the map of 1-skeleta $\phi:\Gamma_Y\to \Gamma_X$, one obtains a factorisation of maps of graphs
\[
\Gamma_Y\to \bar{\Gamma}_Y\to \Gamma_X
\]
where $\Gamma_Y\to \bar{\Gamma}_Y$ is a $\pi_1$-surjection and $\bar{\Gamma}_Y\to \Gamma_X$ is an immersion. The natural map $S_Y \to S_X$ factors through the fibre product of graphs $S_X\times_{\Gamma_X} \bar{\Gamma}_Y$\,, and setting $S_{\bar{Y}}$ to be the image of $S_Y$ in $S_X\times_{\Gamma_X} \bar{\Gamma}_Y$ defines the required complex $\bar{Y}$. 

It is clear that $\phi_0$ is a $\pi_1$-surjection, because it is true on the 1-skeleta by the result for graphs.  The map of 1-skeleta $\bar{\Gamma}_Y\to \Gamma_X$ is an immersion and $S_{\bar{Y}}$ embeds in $\bar{\Gamma}_Y\times_{\Gamma_X} S_X$ by construction, so $\bar{\phi}$ is a branched immersion by Lemma \ref{lem: Immersions and fibre products}.

It remains to prove the universal property. To that end, suppose that $Y\to X$ factors through a branched immersion $Z\to X$.  By Stallings' folding lemma for graphs, the map of 1-skeleta $\bar{\Gamma}_Y$ factors uniquely through $\Gamma_Z\to \Gamma_X$. Since $Z\to X$ is a branched immersion, $S_Z$ can be identified with a subset of the fibre product $\Gamma_Z\times_{\Gamma_X} S_X$.  The diagram
\begin{center}
  \begin{tikzcd}
S_{\bar{Y}}\arrow{r}\arrow{rd}&\bar{\Gamma}_Y\arrow{r} & \Gamma_Z\arrow{d}\\
& S_X \arrow{r} & \Gamma_X
  \end{tikzcd}
\end{center}
commutes so, by the universal property of fibre products, there is a canonical map $S_{\bar{Y}}\to \Gamma_Z\times_{\Gamma_X} S_X$. Since $S_Y\to S_Z\subseteq \Gamma_Z\times_{\Gamma_X} S_X$ factors through $S_{\bar{Y}}$, the image of $S_{\bar{Y}}$ in $\Gamma_Z\times_{\Gamma_X} S_X$ is contained in $S_Z$, which completes the proof.
\end{proof}

Since branched immersion can be recognised locally, purely by looking at links, we would like to work with them whenever possible. However, we cannot restrict our attention solely to branched immersions, since we sometimes need to unfold in order to analyse the structure of a 2-complex $X$. Maps for which the folded representative is similar to the domain are therefore especially useful to work with. This motivates the definition of an \emph{essential map}.

\begin{definition}[Essential maps]\label{def: Essential map}
A branched morphism $\phi:Y\to X$ is an \emph{essential equivalence} if it satisfies the following two conditions:
\begin{enumerate}[(i)]
\item the map of 1-skeleta $\phi:\Gamma_Y\to\Gamma_X$ is a homotopy equivalence;
\item the map $\phi:S_Y\to S_X$ is an isomorphism.
\end{enumerate} 
More generally, a branched morphism $\phi:Y\to X$ is \emph{essential} if the folded map $\phi_0:Y\to\bar{Y}$ provided by Lemma \ref{lem: Folding 2-complexes} is an essential equivalence.
\end{definition}

\subsection{Irreducible 2-complexes}

With the definition of an essential map in hand, we are ready to define the sets $\Irred(X)$ and $\Surf(X)$ mentioned in the introduction.  We start with $\Surf(X)$, which is slightly easier to characterise.

\begin{definition}\label{def: Surf(X)}
A branched 2-complex $Y$ is said to be a \emph{surface} if its realisation is homeomorphic to a surface, or equivalently if $\Lk_Y(v)$ is a circle for every vertex $v$.  For any 2-branched complex $X$, the set $\Surf(X)$ consists of all essential maps $Y\to X$, where $Y$ is a finite (but not necessarily connected) surface.
\end{definition}

The set $\Irred(X)$ is slightly more complicated to define. We give a definition in terms of links of vertices.

\begin{definition}\label{def: (Visibly) irreducible 2-complex}
A branched 2-complex $Y$ is said to be \emph{visibly irreducible} if, for every vertex $v$ of $Y$, the link $\Lk_Y(v) $ satisfies the following conditions:
\begin{enumerate}[(i)]
\item $\Lk_Y(v)$ is finite with at least two vertices;
\item $\Lk_Y(v)$ is a core graph, i.e.\ every vertex has valence at least 2;
\item $\Lk_Y(v)$ is connected;
\item $\Lk_Y(v)$ has no cut vertices.
\end{enumerate}
\end{definition}

Roughly speaking, the point of this definition is that any finite 2-complex can be unfolded to a wedge of a graph, some discs, and some visibly irreducible factors. The reader is referred to \cite[\S4]{wilton_rational_2024} for a full discussion.

The set $\Irred(X)$ can then be characterised in the same way as $\Surf(X)$.

\begin{definition}\label{def: Irred(X)}
Let $X$ be a branched 2-complex. The set $\Irred(X)$ consists of all essential maps $Y\to X$, where $Y$ is finite and visibly irreducible, but not necessarily connected.
\end{definition}

\subsection{\texorpdfstring{$\Pi$}{Pi}-complexes and origamis}

To handle the two definitions of $\Surf(X)$ and $\Irred(X)$ uniformly, we can make the following common generalisation.

\begin{definition}\label{def: Essential properties}
A set $\Pi$ of graphs is called \emph{suitable} if $\Pi$ is closed under isomorphism and every graph in $\Pi$ is finite, connected and contains at least one edge. 
For any suitable set of graphs $\Pi$, a \emph{$\Pi$-complex} is a finite branched 2-complex $Y$ such that the link of every vertex of $Y$ is in $\Pi$. The set $\Ess_\Pi(X)$ consists of all essential maps $Y\to X$, where $Y$ is a finite, but not necessarily connected, $\Pi$-complex.
\end{definition}

Both $\Surf(X)$ and $\Irred(X)$ are instances of $\Ess_\Pi(X)$ for appropriate choices of suitable set $\Pi$, as recorded in the following lemma.

\begin{lemma}\label{lem: Surf and Irred are suitable}
Let $X$ be a branched 2-complex.
\begin{enumerate}[(i)]
\item Let $\Sigma$ be the set of graphs homeomorphic to the circle. Then $\Surf(X)=\Ess_\Sigma(X)$.
\item Let $\Xi$ be the set of connected, finite core graphs with at least two vertices and without cut vertices. Then $\Irred(X)=\Ess_\Xi(X)$.
\end{enumerate}
\end{lemma}
\begin{proof}
This is immediate from the definitions of $\Surf(X)$ and $\Irred(X)$.
\end{proof}

The idea of our main theorem, Theorem \ref{thm: General rationality theorem}, is to encode the elements of $\Ess_\Pi(X)$ as the integer points of a suitable linear system. In fact, the linear system encodes not just $\Ess_\Pi(X)$, but elements of $\Ess_\Pi(X)$ equipped with origamis, which in turn certify that the relevant map is indeed essential. To use origamis to certify that a map of branched 2-complexes is essential, we use the following definition, which adapts the context of origamis from graphs to complexes.

\begin{definition}\label{def: Origamis and quotient 2-complexes}
Let $Y$ be a branched 2-complex. An \emph{origami} $\Omega$ on $Y$ is an origami on the 1-skeleton $\Gamma_Y$. The \emph{quotient branched 2-complex} $Y/\Omega$ is now defined naturally as follows:
\begin{enumerate}[(i)]
\item the 1-skeleton  is the quotient graph $\Gamma_Y/\Omega$;
\item the faces are the same as the faces of $Y$, attached using the composition $S_Y\to \Gamma_Y\to \Gamma_Y/\Omega$.
\end{enumerate}
There is a natural quotient map $q:Y\to Y/\Omega$.  Item (ii) of Definition \ref{def: Essential map} is satisfied by construction, so  $q$ is essential if the origami $\Omega$ is essential, by Theorem \ref{thm: Essential origamis}.
\end{definition}

As in the case of graphs, we want to use origamis to certify that a given map is essential, using the notion of a \emph{compatible} origami by extending Definition \ref{def: Compatible origami}.

\begin{definition}\label{def: Origamis compatible with maps of complexes}
Let $\phi:Y\to X$ be a morphism of branched 2-complexes. An origami $\Omega$ on $Y$ is said to be \emph{compatible with $\phi$} if $\phi$ factors through the quotient map $Y\to Y/\Omega$, and the factor map $Y/\Omega\to X$ is a branched immersion. 
\end{definition}

By Theorem \ref{thm: Essential origamis}, a morphism is essential if and only if there is a compatible essential origami.  Just as in the case of graphs, it is easy to certify that an origami is compatible with a given morphism.

\begin{lemma}\label{lem: Certifying compatibility for complexes}
A morphism of branched 2-complexes $\phi:Y\to X$  is compatible with an origami $\Omega$ if and only if items (i)-(iii) of Remark \ref{rem: Certifying compatibility} are satisfied, together with the following additional condition:
\begin{enumerate}[(i)]\setcounter{enumi}{3}
\item  if $u_1,u_2$ are vertices of $S_Y$ with $\phi(u_1)=\phi(u_2)$ and $w_Y(u_1),w_Y(u_2)$ are in the same component of $V_\Omega$ then $u_1=u_2$.
\end{enumerate}
\end{lemma}
\begin{proof}
The conditions of Remark \ref{rem: Certifying compatibility} characterise when $\phi$ factors through the quotient map $Y\to Y/\Omega$ and the factor map  $Y/\Omega\to X$ is an immersion on the 1-skeleton. Item (iv) above is a restatement of the condition that 
\[
S_Y\to (\Gamma_Y/\Omega)\times_{\Gamma_X} S_X
\]
should be an embedding. Therefore, by Lemma \ref{lem: Immersions and fibre products}, $\phi$ is compatible with $\Omega$ if and only if the conditions of Remark \ref{rem: Certifying compatibility}, together with item (iv), are satisfied.
\end{proof}

Since we will use essential origamis to certify that our morphisms are essential, we next refine $\Ess_\Pi(X)$ to include this extra information.

\begin{definition}\label{def: Orig_Pi}
Fix a branched 2-complex $X$. The notation $\Orig_\Pi(X)$ denotes the set of morphisms $\phi:Y\to X$, where $Y$ is a finite $\Pi$-complex equipped with an essential origami compatible with $\phi$. There is a natural notion of isomorphism on $\Orig_\Pi(X)$: two morphisms $\phi_1:Y_1\to X$ and $\phi_2:Y_2\to X$, equipped with compatible essential origamis, are isomorphic if there is an isomorphism of 2-complexes $\theta:Y_1\to Y_2$ such that $\phi_1=\phi_2\circ\theta$ and such that $e\sim_Oe'$ if and only if $\theta(e)\sim_O\theta(e')$, for each pair of edges $e,e'$ of $Y_1$. We will always consider the elements of $\Orig_\Pi(X)$ up to this notion of isomorphism.
\end{definition}

The relationship between $\Orig_\Pi(X)$ and $\Ess_\Pi(X)$ can be conveniently summarised in the following proposition.

\begin{proposition}\label{prop: Forgetting origamis}
The map $\Orig_\Pi(X)\to\Ess_\Pi(X)$ given by forgetting the origami on $Y$ is surjective and finite-to-one.
\end{proposition}
\begin{proof}
The forgetful map is well defined and surjective by Theorem \ref{thm: Essential origamis}. It is finite-to-one because there are only finitely many origamis on any finite graph.
\end{proof}

\section{A linear system}\label{sec: A linear system}

Fix a finite branched 2-complex $X$. Our goal is now to define a rational cone $C(\R)$ in some real vector space $\R^N$ such that the integer points $C(\Z)=C(\R)\cap\Z^N$ correspond to the elements of $\Orig_\Pi(X)$. The construction can be summarised in the following theorem.

\begin{theorem}\label{thm: Rational cone}
Let $X$ be a finite branched 2-complex and $\Pi$ a set of connected finite graphs, each with at least one edge.
There is an explicitly defined rational cone $C(\R)$ and a surjection
\[
\Phi:\Orig_\Pi(X)\to C(\Z)\,
\]
such that the preimage of any vector in $C(\Z)$ is a finite set of isomorphism classes. 
\end{theorem}

We will usually abuse notation and let $Y$ denote an element of $\Orig_\Pi(X)$, i.e.\ a morphism $Y\to X$ equipped with an essential origami $\Omega$.

The rest of this section is devoted to the construction of the cone $C(\mathbb{R})$ and the proof of Theorem \ref{thm: Rational cone}.

\subsection{Vertex and edge blocks}

The idea is to decompose an element of $\Orig_\Pi(X)$ into \emph{vertex blocks}, of which there should only be finitely many combinatorial types. The map $\Phi$ then counts the number of vertex blocks. We next give the definition of a vertex block, which is extremely involved. Intuitively, a vertex block encodes the vertex space of the graph of graphs associated to an origami. The construction is phrased in terms of the link graphs $\Lk_X(x)$. Since the path components of a link correspond to vertices of the complex, the $\pi_0$ functor will play a role, and we will use the notation $\pi_0(v)$ to denote the path component of a vertex $v$.

\begin{definition}[Vertex block]\label{def: Vertex blocks}
A \emph{vertex block} over $X$, denoted by $\beta$, consists of the following data:
\begin{enumerate}[(a)]
\item a vertex $x_\beta$ of $X$;
\item  two finite graphs $L(\beta)$ and $\bar{L}(\beta)$, together with morphisms $L(\beta)\to \bar{L}(\beta)\to \Lk_X(x_\beta)$; and
\item  equivalence relations $\sim_O$ and $\sim_C$ on the vertices of $L(\beta)$.
\end{enumerate}

The equivalence relations $\sim_O$ and $\sim_C$ together define two graphs -- a \emph{vertex graph} $V_\beta$ and an \emph{edge graph} $E_\beta$ -- as in the definition of an origami.

The vertices of the edge graph $E_\beta$ are of two kinds: the $\sim_O$-equivalence classes $[v]_O$ and the $\sim_C$-equivalence classes $[v]_C$; the edges of $E_\beta$ are the vertices of $L(\beta)$, and each edge $v$ joins $[v]_O$ to $[v]_C$.

The vertex graph $V_\beta$ also has two kinds of vertices: the path components of $L(\beta)$ and the $\sim_C$-equivalence classes $[v]_C$. The edges are again the vertices of $L(\beta)$, with an edge $v$ joining a path component $\pi_0(v)$ to the equivalence class $[v]_C$.

The data of the vertex block is then subject to the following conditions:
\begin{enumerate}[(i)]
\item $L(\beta)$ is a disjoint union of graphs in $\Pi$;
\item $L(\beta)\to \bar{L}(\beta)$ is bijective on edges;
\item $\bar{L}(\beta)\to \Lk_X(x_\beta)$ is injective;
\item the vertex graph $V_\beta$ is a tree and the edge graph $E_\beta$ is a forest;
\item if three vertices $v_1,v_2,v$ of $L(\beta)$ satisfy $v_1,v_2\in[v]_O$ then $v$ does not separate $\pi_0(v_1)$ from $\pi_0(v_2)$ in $V_\beta$;
\item if two vertices $v_1$ and $v_2$ of $L(\beta)$ have the same image in $\Lk_X(x_\beta)$ then they are in the same path component of $E_\beta$.
\end{enumerate}
\end{definition}

Let $\vblocks=\vblocks(X)$ be the set of all vertex blocks up to the natural combinatorial notion of isomorphism.   Although the definition of $\vblocks$ is extremely complicated, it is not hard to see that $\vblocks$ is finite, as long as $X$ is also finite.

\begin{remark}\label{rem: Finitely many blocks}
If $X$ is a finite branched 2-complex then $\vblocks(X)$ is finite. Indeed, since $X$ is finite there are only finitely many links $\Lk_X(x_\beta)$, each of which is a finite graph. Since $\bar{L}(\beta)$ injects into $\Lk_X(x_\beta)$, it follows that there are only finitely many combinatorial types of maps $\bar{L}(\beta)\to \Lk_X(x_\beta)$.  Being a disjoint union of graphs in $\Pi$, the graph $L(\beta)$ has no isolated vertices, so the number of vertices is at most twice the number of edges, which in turn is bounded since $L(\beta)\to \bar{L}(\beta)$ is bijective on edges. Therefore $L(\beta)$ is of bounded size, and so there are only finitely many maps $L(\beta)\to \bar{L}(\beta)$ up to combinatorial equivalence. Finally, because $L(\beta)$ is a finite graph, there are only finitely many equivalence relations $\sim_O$ and $\sim_C$ on the vertices of $L(\beta)$.
\end{remark}

The vertex blocks in $\vblocks$ are the variables of our linear-programming problem. To this end, we will work in the finite-dimensional vector space $\R^\vblocks$, in which a typical vector is denoted by $(t_\beta)=(t_\beta)_{\beta\in\vblocks}$.  The idea is that an essential map $Y\to X$ (equipped with an origami) is built out of blocks, and the vector $(t_\beta)$ counts the number of such blocks of isomorphism type $\beta$. This idea can be made precise by the notion of an \emph{induced vertex block}.

\begin{definition}[Induced vertex block]\label{def: Induced vertex block}
Let  $Y$ be a $\Pi$-complex, let $\phi:Y\to X$ be a morphism and let $\Omega$ be an essential origami on $Y$, compatible with $\phi$.  That is, the map $\phi$ factors as
\[
Y\stackrel{\phi_0}{\to} \bar{Y}\stackrel{\bar{\phi}}{\to} X
\]
where $\bar{Y}=Y/\Omega$ and $\phi_0$ is the quotient map. For any vertex $\bar{u}$ of $\bar{Y}$, the \emph{induced vertex block} $\beta=\beta(\bar{u})$ is defined as follows:
\begin{enumerate}[(a)]
\item $x_\beta:=\bar{\phi}(\bar{u})$ ;
\item  $\bar{L}(\beta):=\Lk_{\bar{Y}}(\bar{u})$, while $L(\beta):=\coprod_{u\in\phi_0^{-1}(\bar{u})}\Lk_Y(u)$, and the map $L(\beta)\to\bar{L}(\beta)$ is the coproduct of the induced maps $(d\phi_0)_u$ on links;
\item  the equivalence relations $\sim_O$ and $\sim_C$ on the vertices of $L(\beta)$ are restricted from $\Omega$ in the natural way.
\end{enumerate}
We need to check that the conditions of Definition \ref{def: Vertex blocks} are satisfied:
\begin{enumerate}[(i)]
\item the fact that $Y$ is a $\Pi$-complex implies that every path component of $L(\beta)$ is in $\Pi$;
\item $L(\beta)\to \bar{L}(\beta)$ is bijective on edges by the definition of the quotient complex $Y/\Omega$;
\item $\bar{\phi}$ is a branched immersion, and so the induced map on links $d\bar{\phi}_{\bar{u}}:\bar{L}(\beta)\to\Lk_X(x_\beta)$ is injective;
\item the vertex graph $V_\beta$ is a component of  $V_\Omega$ and the edge graph $E_\beta$ is a union of components of $E_\Omega$, so $V_\beta$ is a tree and $E_\beta$ is a forest because $\Omega$ is essential;
\item the local consistency of $\Omega$ is equivalent to item (v) of Definition \ref{def: Vertex blocks};
\item the final condition, that two vertices $v_1$ and $v_2$ of $L(\beta)$ with the same image in $\Lk_X(x_\beta)$ are in the same path component of $E_\beta$, follows immediately from the fact that $\Omega$ is compatible with $\phi$, by item (iii) of Remark \ref{rem: Certifying compatibility}.
\end{enumerate}
\end{definition}

The equations that cut out the subspace $V$ from the vector space $\R^\vblocks$ make use of \emph{edge blocks}.

\begin{definition}[Edge block]\label{def: Edge blocks}
An \emph{edge block} over $X$, denoted by $\gamma$, consists of the following data:
\begin{enumerate}[(a)]
\item an edge $e_\gamma$ of $X$;
\item a (possibly empty) finite subset $L(\gamma)$ of $\Lk_X(e_\gamma)$;
\item a partition $V(\gamma)$ of $L(\gamma)$; and
\item a pair of equivalence relations $\sim_O$ and $\sim_C$ on $V(\gamma)$.
\end{enumerate}
\end{definition}

The data of an edge block would enable us to define an edge graph, analogously to the edge graphs that appear in origamis and vertex blocks, and we could impose the condition that this edge graph should be a tree. However, we will not need to do that here. Let $\eblocks\equiv\eblocks(X)$ denote the set of edge blocks up to the natural notion of combinatorial isomorphism. Similarly to the case of vertex blocks, it is not hard to see that $\eblocks$ is finite when $X$ is, although again we will not need that here.

Edge blocks encode the part of a vertex block that lies above an edge. The next definition makes this precise.

\begin{definition}[Induced edge block]\label{def: Induced edge block}
For each vertex block $\beta\in\vblocks$ and each edge $e$ such that $\iota(e)=x_\beta$, the \emph{induced edge block} $\gamma_\beta(e)$ is defined in the following way.
\begin{enumerate}[(a)]
\item The edge $e_{\gamma_\beta(e)}$ is defined to be $e$.
\item The set $L(\gamma_\beta(e))$ is the intersection of the image of $L(\beta)$ with $\Lk_X(e)$ in $\Lk_X(x_\beta)$.
\item The partition $V(\gamma_\beta(e))$ on $L(\gamma_\beta(e))$ is induced by the links of vertices in $L(\beta)$, so each element is of the form $L(\gamma_\beta(e))\cap\Lk_{L(\beta)}(v)$, for $v$ a unique vertex of $L(\beta)$.
\item The equivalence relation $\sim_O$ and $\sim_C$ on $V(\gamma_\beta(e))$ are defined as follows. The open equivalence relation is now pulled back from $\beta$, so 
\[
L(\gamma_\beta(e))\cap\Lk_{L(\beta)}(v_1)\sim_O L(\gamma_\beta(e))\cap\Lk_{L(\beta)}(v_2)
\]
if and only if $v_1\sim_O v_2$ in $L(\beta)$. The closed equivalence relation is pulled back in the same way, so
\[
L(\gamma_\beta(e))\cap\Lk_{L(\beta)}(v_1)\sim_C L(\gamma_\beta(e))\cap\Lk_{L(\beta)}(v_2)
\]
if and only if $v_1\sim_C v_2$ in $L(\beta)$.
\end{enumerate}
\end{definition}

Finally, to define the cone $C(\mathbb{R})$, we use the fact that the orientation-reversing involution on edges extends to edge blocks defined over them.

\begin{definition}[Opposite edge block]\label{def: Opposite edge blocks}
Let $\gamma\in\eblocks$ be an edge block. The \emph{opposite edge block} $\opp{\gamma}$ is defined as follows:
\begin{enumerate}[(a)]
\item $e_{\opp{\gamma}}:=\opp{e}_\gamma$;
\item $L(\opp{\gamma})$ is the image of $L(\gamma)$ under the canonical orientation-reversing bijection $\Lk_X(e)\to \Lk_X(\opp{e})$;
\item  $V(\opp{\gamma})$ is the induced image of $V(\gamma)$ under the orientation-reversing bijection $L(\opp{\gamma})\to L(\gamma)$, and  indeed we write $\opp{v}$ for the induced image of a subset $v$ of $L(\gamma)$.
\item finally, the roles of $\sim_O$ and $\sim_C$ are reversed, so $\opp{u}\sim_O\opp{v}$ in $V(\opp{\gamma})$ if and only if $u\sim_C v$ in $V(\gamma)$, and likewise $\opp{u}\sim_C\opp{v}$ in $V(\opp{\gamma})$ if and only if $u\sim_O v$ in $V(\gamma)$.
\end{enumerate}
\end{definition}

\subsection{The rational cone}

The rational cone $C(\R)$ can be constructed in a way that is now fairly standard; the reader may like to compare the construction given here with the constructions of \cite{calegari_stable_2009,louder_uniform_2024,wilton_essential_2018}.  We work in  the real vector spaces $\R^{\vblocks}$.  For each edge block $\gamma_0$ over an edge $e_0$ of $X$, the corresponding \emph{gluing equation} constrains the vector $(t_\beta)\in \R^{\vblocks}$ by insisting that
\[
\sum_{\gamma_\beta(e_0)=\gamma_0} t_\beta = \sum_{\gamma_\beta(\opp{e}_0)=\opp{\gamma}_0} t_{\beta}\,,
\]
and $C(\R)$ is the set of vectors in the positive orthant $\R_{\geq 0}^{\vblocks}$ that satisfy all the gluing equations. Note that the gluing equations are linear with integer coefficients, so $C(\mathbb{R})$ is indeed a rational cone. 

The map $\Phi$  from $\Orig_\Pi(X)$ to the set of integer points $C(\Z)$ can be easily defined using induced vertex blocks.  A typical element of $\Orig_\Pi(X)$ consists of a morphism $\phi:Y\to X$ and a compatible essential origami $\Omega$. Now $\Phi$  counts the number of vertices of the folded representative $\bar{Y}$ for which the induced vertex block is of each isomorphism type. More precisely, $\Phi$ sends the given element to the vector $(t_\beta)$, where
\[
t_\beta=\#\{\bar{u}\in\bar{Y}\mid \beta(\bar{u})=\beta\}\,.
\]
The proof of Theorem \ref{thm: Rational cone} is now tedious but routine.

\begin{proof}[Proof of Theorem \ref{thm: Rational cone}]
We need to check that the image of $\Phi$ satisfies the gluing equations, that $\Phi$ is surjective, and that each vector has only finitely many isomorphism types in its preimage under $\Phi$.

To see that $\Phi$ satisfies the gluing equations, consider an edge $e$ of $\bar{Y}$, let $\bar{u}=\iota(e)$ and let $\bar{v}=\tau(e)$. By construction, the opposite block of $\gamma_{\beta(\bar{u})}(e)$ is precisely $\gamma_{\beta(\bar{v})}(\opp{e})$. Thus, there is a bijection between the set of vertices of $\bar{Y}$ whose vertex block induces a given edge block $\gamma$ and the set  of vertices of $\bar{Y}$ whose vertex block induces the opposite edge block $\opp{\gamma}$. From this, it follows that the image of $\Phi$ satisfies the gluing equations.

To see that $\Phi$ is surjective, consider an arbitrary vector $(t_\beta)$ in $C(\Z)$; we need to construct a morphism $\phi:Y\to X$ and a compatible essential origami $\Omega$ that $\Phi$ sends to $(t_\beta)$. To this end, take $t_\beta$ copies of each vertex block $\beta\in \vblocks$. Since $(t_\beta)$ satisfies the gluing equations, for each pair of edges $\{e,\opp{e}\}$ of $X$ and for each edge block $\gamma$ over $e$, we may choose a bijection between the vertex blocks inducing $\gamma$ over $e$ and the vertex blocks inducing $\opp{\gamma}$ over $\opp{e}$.  For each such pair of vertex blocks $\beta$ and $\beta'$, choose an isomorphism between $\gamma_\beta(e)$ and $\gamma_{\beta'}(\opp{e})$.  This provides the data to construct $Y$ uniquely, as follows: the disjoint union of the links of vertices in $Y$ is the union of the graphs $L(\beta)$ across our chosen set of vertex blocks, and the chosen isomorphisms provide bijections between the link of each edge of $Y$ and the link of its opposite. By Remark \ref{rem: Extra structure on links}, this determines $Y$. The data of the blocks determines $\phi:Y\to X$, and the equivalence relation $\sim_O$ on the vertex blocks defines an origami $\Omega$ compatible with $\phi$. Note that $\Omega$ is essential by construction. This completes the proof that $\Phi$ is surjective.

Finally, preimages under $\Phi$ consist of finitely many isomorphism types because, in the above construction, the isomorphism type was determined by the choices of finitely many possible bijections and isomorphisms.
\end{proof}

\section{Curvature invariants}\label{sec: Curvature invariants}

In the next section we give the true definitions of our curvature invariants. In the subsequent section, we observe that these can be expressed in terms of linear functions on the cone $C(\R)$.

\subsection{Curvature invariants}

We start with a natural extension of the notion of Euler characteristic to the context of branched complexes.

\begin{definition}[Total and average curvatures]\label{def: Total and average curvatures}
The \emph{area} of a branched 2-complex $X$ is the sum of the areas of its faces. The \emph{total curvature} of a 2-complex $X$ is the quantity
\[
\tau(X):=\Area(X)+\chi(\Gamma_X)
\]
where $\chi(\Gamma_X)$ denotes the usual Euler characteristic of the 1-skeleton of $X$, i.e\ the the number of vertices minus the number of edges. 

The \emph{average curvature} of $X$ is the quantity
\[
\kappa(X):=\frac{\tau(X)}{\Area(X)}.
\]
\end{definition}

In particular, if $X$ is standard then $\tau(X)$ is the usual Euler characteristic. The definition of $\kappa(X)$ is motivated by the Gauss--Bonnet theorem, which implies that for a Riemannian metric on a closed surface $S$, the average Gaussian curvature is $2\pi\chi(S)/\Area(S)$.

As explained in the introduction, the invariants $\rho_\pm(X)$ are defined by extremising $\kappa$ over $\Irred(X)$. 

\begin{definition}[Irreducible curvature bounds]\label{def: Irreducible curvatures}
For any branched 2-complex $X$,
\[
\rho_+(X):=\sup_{Y\in\Irred(X)} \kappa(Y)\,,
\]
and
\[
\rho_-(X):=\inf_{Y\in\Irred(X)} \kappa(Y)\,.
\]
\end{definition}

The invariants $\sigma_\pm(X)$ are defined similarly, with $\Surf(X)$ replacing $\Irred(X)$.

\begin{definition}[Surface  curvature bounds]\label{def: Surface curvatures}
For any branched 2-complex $X$,
\[
\sigma_+(X):=\sup_{Y\in\Surf(X)} \kappa(Y)\,,
\]
and
\[
\sigma_-(X):=\inf_{Y\in\Surf(X)} \kappa(Y)\,.
\]
\end{definition}

\subsection{Linear functions}

Throughout this section, $X$ is a finite branched 2-complex, $\Pi$ is a suitable collection of graphs, $C(\R)$ is the cone provided by Theorem \ref{thm: Rational cone} and $\Phi$ is the surjection provided by the same theorem. The final ingredient of the main theorem is the observation that both the total curvature and the area are linear functions on $C(\mathbb{R})$. We start with area.

\begin{lemma}\label{lem: Area functional}
There is a linear function $\Area$ on $C(\mathbb{R})$ such that
\[
\Area\circ\Phi(Y)=\Area(Y)
\]
for any $Y\in \Orig_\Pi(X)$.
\end{lemma}
\begin{proof}
For each face $f$ of $X$, let $l(f)$ be the number of oriented edges in $f$. (This is exactly twice the length of $f$ thought of as a cyclic graph, since in a Serre graph each edge appears with both orientations.) Each edge $e$ of $L(\beta)$, where $\beta$ is a vertex block, is equipped with a map to $\Lk_X(x_\beta)$, sending $e$ into some face $f(e)$ of $X$. Define the linear map $\Area$ on the basis vector $\delta_\beta$ corresponding to $\beta$ by
\[
\Area(\delta_\beta)=\sum_{e\in E_{L(\beta)}} \Area(f(e))/l(f(e))\,.
\]
Now suppose that $\Phi(Y)=(t_\beta)$.  Then
\[
\Area\circ\Phi(Y)=\sum_{\beta\in\vblocks}t_\beta\sum_{e\in E_{L(\beta)}} \Area(f(e))/l(f(e))\,.
\]
By construction, the edge set of $S_Y$ consists of exactly $t_\beta$ copies of each edge of $L(\beta)$, and so this sum can be rearranged as
\[
\Area\circ\Phi(Y)=\sum_{e\in E_{S_Y}} \Area(f(e))/l(f(e))\,.
\]
For each face $f$ of $X$, the number of edges of $S_Y$ mapping into $f$ is exactly $\sum_{\phi(f')=f} m_{\phi}(f')l(f)$. Therefore, this sum becomes
\begin{eqnarray*}
\Area\circ\Phi(Y)&=&\sum_{f\in F_X}\sum_{\phi(f')=f} m_{\phi}(f')l(f).\Area(f)/l(f)\\
&=&\sum_{f\in F_X}\sum_{\phi(f')=f} m_{\phi}(f').\Area(f)\\
&=&\sum_{f'\in F_Y}\Area(f')\\
&=&\Area(Y)
\end{eqnarray*}
as claimed.
\end{proof}

The next lemma makes a similar observation for the Euler characteristic of the 1-skeleton $\Gamma_Y$.

\begin{lemma}\label{lem: Euler characteristic functional}
There is a linear function $\chi$ on $C(\mathbb{R})$ such that
\[
\chi\circ\Phi(Y)=\chi(\Gamma_Y)
\]
for any $Y\in \Orig_\Pi(X)$.
\end{lemma}
\begin{proof}

For $\beta\in\vblocks$ and $\delta_\beta$ the corresponding basis vector of $C(\R)$, let
\[
\chi(\delta_\beta)=\#\pi_0(L(\beta))-\#V_{L(\beta)}/2\,.
\]
Again writing $\Phi(Y)=(t_\beta)$, the links of $Y$ consist of $t_\beta$ copies of the graphs $L(\beta)$, with the path components of $L(\beta)$ corresponding to the vertices of $Y$ (since $\Pi$ is suitable) and the vertices of $L(\beta)$ corresponding to the edges of $Y$. With the above definition of the function $\chi$, the desired equation
\[
\chi\circ\Phi(Y)=\#V_Y-\#E_Y/2=\chi(\Gamma_Y)
\]
now follows immediately.
\end{proof}

The corresponding statement for the total curvature $\tau(Y)$ follows immediately.

\begin{lemma}\label{lem: Total curvature functional}
There is a linear function $\tau$ on $C(\mathbb{R})$ such that
\[
\tau\circ\Phi(Y)=\tau(Y)
\]
for any $Y\in \Orig_\Pi(X)$.
\end{lemma}
\begin{proof}
Set $\tau=\Area+\chi$.
\end{proof}

With these facts in hand, we can prove the most general form of the rationality theorem.

\begin{theorem}\label{thm: General rationality theorem}
Let $X$ be a finite branched 2-complex and let $\Pi$ be any suitable set of graphs. If $\Ess_\Pi(X)$ is non-empty then:
\begin{enumerate}[(i)]
\item\label{item: Linear programming} the supremum and infimum of $\kappa(Y)$ over all $Y\in\Ess_\Pi(X)$ are the extrema of an explicit linear-programming problem; furthermore,
\item\label{item: Attained} the supremum is attained by some $Y_{\max}\in\Ess_\Pi(X)$ and the infimum is attained by some $Y_{\min}\in\Ess_\Pi(X)$.
\end{enumerate}
\end{theorem}
\begin{proof}
We give the proof for the supremum, since the proof for the infimum is identical.

Since $\kappa$ does not depend on any choice of origami, the supremum is also the supremum of $\kappa(Y)$ over all $Y\in\Orig_\Pi(X)$, by Proposition \ref{prop: Forgetting origamis}. By Lemmas \ref{lem: Area functional} and \ref{lem: Total curvature functional}, the projective function
\[
\kappa=\tau/\Area
\]
on $C(\R)\smallsetminus 0$ is such that $\kappa(Y)=\kappa\circ\Phi(Y)$ for each $Y\in\Orig_\Pi(X)$. On the compact, rational polytope
\[
P=C(\R)\cap\{\Area(t)=1\}
\]
the simplex algorithm implies that the supremum of $\kappa$ is achieved at a vertex of $P$, equal to a rational vector $t_{\max}$. Thus, by construction, $\kappa(t_{\max})$ is the desired supremum, which proves item (\ref{item: Linear programming}). Since $t_{\max}$ is rational and $\kappa$ is projective, we may replace $t_{\max}$ by a multiple so that $t_{\max}\in C(\Z)$. By the surjectivity guaranteed by Theorem \ref{thm: Rational cone}, there is $Y_{\max}\in\Orig_\Pi(X)$ such that $\Phi(Y_{\max})=t_{\max}$, so the supremum is attained, proving item (\ref{item: Attained}) for the supremum.
\end{proof}

Since $\Irred(X)$ and $\Surf(X)$ are instances of $\Ess_\Pi(X)$, Theorems \ref{introthm: Rationality for irreducible curvature} and \ref{introthm: Rationality for surface curvature} follow immediately by Lemma \ref{lem: Surf and Irred are suitable}.

\bibliographystyle{plain}

\Addresses

\end{document}